\documentclass[10pt]{article}

\usepackage{authblk}
\usepackage[dvips]{graphicx}
\usepackage{amsfonts}
\usepackage{amsopn}
\usepackage{amsthm} 
\usepackage{amssymb}
\usepackage{color}
\usepackage{hyperref}
\usepackage{soul}

\usepackage{latexsym}
\usepackage{amsgen} 
\usepackage{bigints}

\newcommand\reallywidehat[1]{%
\savestack{\tmpbox}{\stretchto{%
  \scaleto{%
    \scalerel*[\widthof{\ensuremath{#1}}]{\kern-.6pt\bigwedge\kern-.6pt}%
    {\rule[-\textheight/2]{1ex}{\textheight}}
  }{\textheight}%
}{0.5ex}}%
\stackon[1pt]{#1}{\tmpbox}%
}

\newtheorem*{theorem*}{Theorem}
\newtheorem{theorem}{Theorem}[section]
\newtheorem{definition}[theorem]{Definition}
\newtheorem{lemma}[theorem]{Lemma}
\newtheorem{corollary}[theorem]{Corollary}
\newtheorem{proposition}[theorem]{Proposition}

\newtheorem{example}[theorem]{Example}
\newtheorem{remark}[theorem]{Remark}

\newtheorem*{ack*}{Acknowledgment}
\numberwithin{equation}{section}

\newcommand{\cA}{ \mathcal{A}}

\newcommand{\cC}{ \mathcal{C}}

\newcommand{\cF}{ \mathcal{F}}
\newcommand{\cH}{ \mathcal{H}}

\newcommand{\cJ}{ \mathcal{J}}
\newcommand{\cL}{ \mathcal{L}}

\newcommand{\cO}{ \mathcal{O}}
\newcommand{\cP}{ \mathcal{P}}

\newcommand{\cS}{ \mathcal{S}}
\newcommand{\cU}{ \mathcal{U}}
\newcommand{\cV}{ \mathcal{V}}
\newcommand{\cW}{ \mathcal{W}}
\newcommand{\mW}{ \mathcal{W}}
\newcommand{\mP}{ \mathcal{P}}

\newcommand{\bB}{\mathbb B}

\newcommand{\bE}{\mathbb E}
\newcommand{\bG}{\mathbb G}

\newcommand{\bP}{\mathbb P}
\newcommand{\bR}{\mathbb R}
\newcommand{\R}{\mathbb R}
\newcommand{\bS}{\mathbb S}

\newcommand{\bV}{\mathbb V}

\newcommand\bigcdot{\mathpalette\bigcdot@{.5}}

\def\R{{\mathbb{R}}}

\def\cH{\mathcal{H}}
\def\cL{\mathcal{L}}
\def\G{\mathbb{G}}

\usepackage{amssymb}

\begin{document}

\title{Master equations with an individual noise on finite state graphs}
\author[1]{Wilfrid Gangbo\thanks{\href{mailto:wgangbo@math.ucla.edu}{wgangbo@math.ucla.edu}}}
\author[1]{Sebastian Munoz\thanks{\href{mailto:sebastian@math.ucla.edu}{sebastian@math.ucla.edu}}}
\author[2]{Jeremy Wu\thanks{\href{mailto:jeremy.wu@umanitoba.ca}{jeremy.wu@umanitoba.ca}}}
\author[1]{Zhaoyu Zhang\thanks{\href{mailto:zhaoyu@math.ucla.edu}{zhaoyu@math.ucla.edu}}}
\affil[1]{Department of Mathematics, UCLA, California, USA}
\affil[2]{Department of Mathematics, University of Manitoba, Canada}


\date{}
\maketitle

\begin{abstract} We develop a classical well-posedness and regularity theory on a finite connected weighted graph for an extended mean field game system, its associated master equation, and a Hamilton--Jacobi--Bellman equation on the probability simplex, all in the presence of an individual noise operator. The geometric structure is inherited from the logarithmic-mean activation functional of discrete optimal transport, under which the entropic Fokker--Planck equation appears as a gradient flow on the graph and the individual noise operator is a bilinear form in the probability vector and the Wasserstein gradient. A central technical step is a quantitative preservation-of-positivity estimate for the discrete continuity equation, which rules out finite-time boundary degeneracy and yields a classical solution theory for the master equation on the open simplex without imposing any boundary condition. As an application, we recover a Nash equilibrium interpretation of the discrete system in terms of Markov chains on the graph. Our setup is inspired by the computational algorithms for optimal mass transport of \cite{chow2012,li_2017} and provides a rigorous well-posedness theory for several of the equations derived in \cite{GaoLiLiu}.
\end{abstract}

\renewcommand{\thefootnote}{}
\renewcommand{\thefootnote}{\arabic{footnote}}

%
%
\section{Introduction}\label{sec:introd}

This paper develops a classical well-posedness and regularity theory on a finite connected weighted graph $\bG=(\bV,\bE,\omega)$ for three equations that arise in the study of mean field games (MFG) with individual noise: the extended MFG system~\eqref{eq:mfg}, its associated master equation~\eqref{eq:master}, and the Hamilton--Jacobi--Bellman equation on the probability simplex $\cP(\bG)$ satisfied by the value function of an action-minimization problem. All three share a common source of analytic difficulty: the natural geometric structure on $\cP(\bG)$, inherited from discrete optimal transport, is degenerate on the boundary of the simplex. Our contribution is a classical solution theory in the relative interior $\cP_0(\G)$, resting on a quantitative lower bound on densities that rules out finite-time boundary degeneracy and on a regularity theory that upgrades the value function of the action-minimization problem to class $C^2$ in both variables. As a byproduct, the master equation is uniquely solvable on the open simplex without imposing any boundary condition on $\partial\cP_0(\G)$. We also use this theory to construct Nash equilibria for a continuous-time Markov-chain MFG on $\bG$. For the rest of this present section, we prefer to postpone the notations to Section~\ref{sec:setup} while presenting the main ideas and results immediately.

The geometric structure on $\cP(\bG)$ that we work with is inherited from the \emph{logarithmic-mean} activation functional $\theta(r,s)=(r-s)/(\log r-\log s)$, introduced for discrete optimal transport in \cite{chow2012,li_2017}. It gives rise to the \emph{logarithmic-mean Hamiltonian}
\[
\cH_\theta(\mu,p):=\frac 1 4 \sum_{(i,j)\in\bE}\theta(\mu^i,\mu^j)(p^{ij})^2,
\]
which serves as our model Hamiltonian and motivates the structural hypotheses under which our results hold. Two features make this choice central to our analysis: the entropic Fokker--Planck equation on $\bG$ appears as a gradient flow on the weighted graph precisely when the underlying geometry is built from $\theta$ \cite{ErbasMW}, and $\theta$ is the weighting under which the discrete individual-noise operator $\Delta_{\text{ind}}$ of \eqref{eq:sep06.2025.2} is a bilinear form in $(\mu,\nabla_\cW\cV)$, mirroring its continuum counterpart. Systems of the form~\eqref{eq:mfg}--\eqref{eq:master} find their motivation in biological and chemical reaction models \cite{GaoLiu,GaoLiLiu2}, where they arise as thermodynamic limits of Poisson counting processes, and in mean field games on finite state spaces \cite{CardialaguetDLL,GangboMMZ22,HuangMC}.

Our approach contrasts with much of the existing literature on graph mean field games. The works \cite{CecchinP,Delarue2020,GMSouza} start from a Markov chain formulation and derive PDEs on the graph, and \cite{BayraktarC,BayraktarC2,BertucciC} study master equations in the presence of Wright--Fisher or common noise; we take the reverse path, beginning with PDEs on the graph and recover a Markov chain formulation in Section~\ref{sec:NE-discreteMFG}. A second, quantitative difference is that most prior works define the graph gradient by $(\phi^i-\phi^j)_{ij}$, implicitly forcing the underlying graph to be complete with uniform weights, whereas our weighted gradient $\big(\sqrt{\omega_{ij}}(\phi^i-\phi^j)\big)_{ij}$ accommodates arbitrary connected graphs with an inhomogeneous weight profile.

Our first result is a quantitative preservation-of-positivity estimate for the discrete continuity equation with a general mobility-dominated flux. It plays a pivotal role throughout the remainder of the paper: it is the tool that makes it possible to do classical analysis inside $\cP_0(\G)$ without any a priori control near $\partial\cP_0(\G)$. The notation used in the statements below---including $\bS(n)$, $\cP_\epsilon(\G)$, $\nabla_\G$, $\Delta_\G$, $\nabla_\cW$, $\Delta_{\text{ind}}$, $\cA_t^T$, and $\cU$---is collected in Sections~\ref{sec:setup} and~\ref{subsection:feb14.2026}.

\begin{theorem}[Quantitative interiority]\label{thm:main-interiority}
Let $h:(0,\infty)\to(0,\infty)$ be a bounded function with
\begin{equation}\label{eq:h properties}
\lim_{u\to 0^+}h(u)=\lim_{u\to +\infty}h(u)=0,
\end{equation}
and let $A\in C^1([0,T]\times(0,1)^n;\bS(n))$ be skew-symmetric and satisfy the mobility-dominated bound
\begin{equation}\label{as:A conti}
|A_{ij}(s,\mu)|\leq (\mu^i+\mu^j)\,h(\mu^j/\mu^i),\qquad (i,j)\in\bE,\ (s,\mu)\in[0,T]\times(0,1)^n.
\end{equation}
Then, for every $\epsilon>0$ and every $\mu\in\cP_\epsilon(\bG)$, the continuity equation
\begin{equation}\label{eq:conti general}
\dot\rho(s)=\nabla_\G\cdot A(s,\rho(s))+\Delta_\G\rho(s),\qquad \rho(0)=\mu,
\end{equation}
admits a unique classical solution $\rho\in C^1([0,T];(0,1)^n)$, and there exist constants $c,r>0$ depending only on $n$, $\omega_{\min}$, $\omega_{\max}$, and $h$ such that
\begin{equation}\label{eq:rho lower bd}
\min_{1\leq i\leq n}\rho_i(s)\geq c\,\epsilon\,e^{-rs},\qquad s\in[0,T].
\end{equation}
\end{theorem}

The model case is $A(s,\mu)=D_p\cH_\theta(\mu,p(s))$ for continuous $p:[0,T]\to\bS(n)$, for which~\eqref{as:A conti} follows from~\eqref{eq:dec21.2025.2}.

Our second main result is classical well-posedness for the forward--backward MFG system~\eqref{eq:mfg} and its associated master equation~\eqref{eq:master} in the \emph{extended} setting \cite{LiSo20,Mun23}, where the flux $B$ need not coincide with $D_p H$. The Lasry--Lions monotonicity conditions~\eqref{as:uniqueness 1}--\eqref{as:uniqueness 3} used for uniqueness are stated in Section~\ref{sec:classical}.

\begin{theorem}[Classical solutions to the MFG system and master equation]\label{thm:main-mfg}
Assume that $H\in C^1((0,1)^n\times\bS(n);\R^n)$ and $B\in C^1((0,1)^n\times\bS(n);\bS(n))$ satisfy, for some constant $C_1>0$, the coercivity and lower bound conditions
\[
(B(\mu,p),p)\geq (H(\mu,p),\mu)-C_1,\qquad H_i(\mu,p)\geq -C_1,
\]
and the mobility-dominated flux bound
\[
|B_{ij}(\mu,p)|\leq (\mu^i+\mu^j)\,a(\mu^j/\mu^i,p) \qquad \text{for all }(i,j)\in\bE,
\]
for some locally bounded $a:(0,\infty)\times\bS(n)\to[0,\infty)$ with $\lim_{u\to 0^+}a(u,p)=\lim_{u\to +\infty}a(u,p)=0$ locally uniformly in $p$, and let $g\in C^1([0,1]^n;\R^n)$ satisfy $|g|\leq C_1$. Then, for every $(t,\mu)\in[0,T)\times\cP_0(\G)$, the extended MFG system
\begin{equation}\label{eq:mfg}
\begin{cases}
\dot\phi(s)=H(\rho(s),\nabla_\G\phi(s))-\Delta_\G\phi(s) & s\in(t,T),\\
\dot\rho(s)=\nabla_\G\cdot B(\rho(s),\nabla_\G\phi(s))+\Delta_\G\rho(s) & s\in(t,T),\\
\phi(T)=g(\rho(T)),\quad \rho(t)=\mu,
\end{cases}
\end{equation}
admits a classical solution $(\phi^{t,\mu},\rho^{t,\mu})\in C^2([t,T];\R^n\times(0,1)^n)$, unique under the Lasry--Lions monotonicity conditions~\eqref{as:uniqueness 1}--\eqref{as:uniqueness 3}. The value function $u:[0,T]\times\cP_0(\G)\to\R^n$ given by $u(t,\mu):=\phi^{t,\mu}(t)$ for $t<T$ and $u(T,\mu):=g(\mu)$ belongs to $C^1([0,T]\times\cP_0(\G);\R^n)$ and is the unique classical solution on $[0,T]\times\cP_0(\G)$ of the master equation
\begin{equation}\label{eq:master}
\begin{cases}
\partial_t u^i-(\nabla_\cW u^i,B(\mu,\nabla_\G u))+\Delta_\text{ind}u^i-H^i(\mu,\nabla_\G u)+\Delta_\G u^i=0 & \text{on }[0,T]\times\cP_0(\G),\\
u(T,\cdot)=g & \text{on }\cP_0(\G).
\end{cases}
\end{equation}
\end{theorem}

Our third main result establishes well-posedness, $C^2$ regularity, and the Hamilton--Jacobi--Bellman (HJB) equation for the value function $\cU$ of~\eqref{eq:aug01.2025.2}, associated with an action-minimization problem driven by a convex Lagrangian on the weighted graph; see Section~\ref{subsection:feb14.2026} for the action $\cA_t^T$, the admissible class $\cC_t^T(\mu,\cdot)$, and the definition of $\cU$.

\begin{theorem}[Classical $C^2$ solutions of the HJB equation on $\cP_0(\G)$]\label{thm:main-hjb}
Let $\cH$ denote the Legendre transform of $\cL$ in its second variable, and suppose that $(\cL,\cF,\cU_T)$ satisfy the structural assumptions collected in Section~\ref{subsection:assumptions-g}. Then:
\begin{enumerate}
\item[(i)] For every $(t,\mu)\in[0,T)\times\cP_0(\G)$, the action-minimization problem defining $\cU(t,\mu)$ admits a unique minimizer $(\rho,m)\in C^\infty([t,T];\cP_0(\G)\times\bS(n))$, characterized by $m=D_p\cH(\rho,-\nabla_\G\phi)-\nabla_\G\rho$, where $(\phi,\rho)\in C^\infty([t,T];\R^n\times\cP_0(\G))$ is the unique classical solution of the forward--backward system
\begin{equation}\label{eq:mfg-withcH}
\begin{cases}
\dot\rho(s)+\nabla_\G\cdot D_p\cH(\rho,-\nabla_\G\phi)-\Delta_\G\rho=0 & s\in(t,T),\\
\dot\phi(s)-D_\mu\cH(\rho,-\nabla_\G\phi)+\Delta_\G\phi-D_\mu\cF(\rho)=0 & s\in(t,T),\\
\phi(T)=D_\mu\cU_T(\rho(T)),\quad \rho(t)=\mu;
\end{cases}
\end{equation}
moreover, $\nabla_\cW\cU(s,\rho(s))=\nabla_\G\phi(s)$ for all $s\in[t,T]$.
\item[(ii)] $\cU$ belongs to $C^2([0,T]\times\cP_0(\G))$ and is a classical solution of the HJB equation
\[
-\partial_t\cU+\cH(\cdot,-\nabla_\cW\cU)-\Delta_{\text{ind}}\cU+\cF=0,\qquad \cU(T,\cdot)=\cU_T.
\]
\end{enumerate}
\end{theorem}

As an application, we turn to a game-theoretic interpretation of the master equation. In Section~\ref{sec:NE-discreteMFG} we show that the solution $u$ of Theorem~\ref{thm:main-mfg} is the value function of a mean field game in which individual players evolve as continuous-time Markov chains on $\bG$: each player selects their transition rate matrix so as to minimize a cost with running term determined by a Lagrangian $L$ and terminal cost $g$, and the Nash-equilibrium rate matrix is read off from $\nabla_\G u$.

\begin{theorem}[Markov-chain Nash equilibria; informal version of Theorem~\ref{thm:main-nash}]
Let $u$ be the value function of Theorem~\ref{thm:main-mfg} and $v^{0,\mu}$ its associated optimal velocity field. Under suitable structural assumptions on the Lagrangian, if $v^{0,\mu}$ is an admissible control for its own density flow, then $u$ and $v^{0,\mu}$ define a Nash equilibrium for a continuous-time Markov-chain mean field game on $\bG$ whose terminal and running costs are determined by $g$ and $L$.
\end{theorem}

Admissibility amounts to a non-negativity condition on the off-diagonal entries of the associated rate matrix, and is expected to hold on sufficiently fine nearest-neighbor discretizations of $\mathbb T^d$; see Remark~\ref{rmk:admissibility-torus}.

Theorem~\ref{thm:main-interiority} is proved in Section~\ref{sec:interiority} by combining a spatial propagation-of-smallness principle (Lemma~\ref{lem:interiority 1}) with a waiting-time estimate (Lemma~\ref{lem:interiority 2}): the former propagates smallness along shortest paths in $\bG$, the latter provides a uniform lower bound on the time to halve the minimum coordinate.

Section~\ref{sec:classical} proves Theorem~\ref{thm:main-mfg}. A priori bounds for the homotopic family~\eqref{eq:mfg lambda}, obtained by combining a Lasry--Lions duality identity with the interiority estimate of Theorem~\ref{thm:main-interiority}, feed into Schaefer's fixed point theorem to yield classical solutions of~\eqref{eq:mfg}; uniqueness follows from discrete Lasry--Lions monotonicity. A $C^1$-dependence-on-initial-data analysis of the flow $(t,\mu)\mapsto(\phi^{t,\mu},\rho^{t,\mu})$ (Proposition~\ref{prop:C1-dependence}) then upgrades this to a classical solution of~\eqref{eq:master}.

Section~\ref{sec:ClassicalHJE} proves Theorem~\ref{thm:main-hjb}. Direct methods yield a unique minimizer for $\cA_t^T$ (Proposition~\ref{prop:action-minimizer}), convex duality combined with Theorem~\ref{thm:main-mfg} characterizes it as the solution of~\eqref{eq:mfg-withcH} (Proposition~\ref{prop:sufficient-min}), and subdifferential analysis gives the gradient identity along the minimizer (Proposition~\ref{thm:main2}), establishing part~(i). For part~(ii), spatial semiconcavity (Proposition~\ref{thm:semi-concave}) together with a time-differentiability argument (Proposition~\ref{prop:1differential}) yields the HJB equation, and a bootstrap through Theorem~\ref{thm:main-mfg} upgrades $\cU$ to $C^2$.

Section~\ref{sec:NE-discreteMFG} proves Theorem~\ref{thm:main-nash}. A variational characterization of $D_{\mu^i}\cH$ as a supremum over velocity perturbations (Lemma~\ref{lem:feb24.2026.2new}), valid under a unique momentum assumption on $(\bar\cL,\cH)$ (see Definition~\ref{defn:feb24.2026.1}), feeds into a martingale verification argument (Lemma~\ref{lem:mar12.2026}) that compares the value function $u$ along a player's Markov-chain trajectory with the cost of any admissible deviation.

\subsection{Notations and preliminaries}\label{sec:setup}
Let $G=(\bV, \bE, \omega)$ denote an undirected graph with vertices $\bV=\{1, \cdots, n\}$ and edges $\bE$, endowed with a ``weight'' $\omega=(\omega_{ij})_{ij}$, i.e. a $n \times n$ symmetric matrix with nonnegative entries $\omega_{ij}$ such that $\omega_{ij}>0$ if and only if $(i,j) \in \bE.$ For simplicity, assume that the graph is connected and simple, with no self-loops or multiple edges. We identify $\bV$ with $\{e_1, \cdots, e_n\},$ the canonical orthonormal basis in $\bR^n$ and for $i \in \bV$, $N(i)$ is the set of $j \in \bV$ such that $(i, j) \in \bE$. Let 
\begin{equation}\label{eq:defn-omega-max-min}
    \omega_{\min}:=\min_{(i,j)\in \bE}\omega_{ij}, \quad \omega_{\max}=\max_{(i,j)\in \bE}\omega _{ij}.
\end{equation}

In standard notation, an $n$--dimensional probability vector is a row vector in $\bR^n$ which we denote by $\mu=(\mu^1, \cdots, \mu^n).$ We denote by $\cP(\bG)$ the probability simplex on $\bG$ and for $\epsilon\geq 0$, we set
\[
\cP_\epsilon(\bG):=\big\{\mu \in \cP(\bG): \mu^i > \epsilon \;\; \forall i \in \bV\big\}.
\]
Let $\bS(n)$ denote the set of skew-symmetric $n\times n$ matrices. The divergence of $(m^{ij})_{ij} \in \bS(n)$ on $\bG$ is the row vector $\nabla_\bG \cdot m \in \bR^n$ defined by $(\nabla_\bG \cdot m)^i=\sum_{j \not= i}^n \sqrt{\omega_{ij}} m^{ji}.$ If $\tilde m \in \bS(n)$, we set
\[
(m,\tilde m ):=\frac{1}{2}\sum_{(i,j)\in \bE} m^{ij} \tilde m^{ij}, \quad \|m\|^2=(m,m ).  
\] 
Here the coefficient $1/2$ accounts for the fact that whenever $(i,j) \in \bE$, then $(j,i) \in \bE$. If $u \in \bR^n,$ then $\nabla_\bG u \in \bR^{n \times n}$ is the skew-symmetric matrix defined by 
\[
(\nabla_\bG u)^{ij}=\sqrt{\omega_{ij}} (u^i-u^j), \qquad \forall i, j \in \bV.
\]
We have the integration by parts formula 
\begin{equation}\label{eq:integration by parts formula} \big(\nabla_\bG \cdot m, u\big)=-\big(m, \nabla_\bG u\big),\quad \forall m \in \mathbb{S}(n),\,\forall u \in \mathbb{R}^n.\end{equation} 
A function $K: \bS(n) \to \bR$ is differentiable at $m^0\in \bS(n)$ if there exists $P \in \bS(n)$ such that 
\[
K(m)-K(m^0)-(m-m^0, P)=o\big(\|m-m^0\|_{\ell_2}\big).
\]
We call $P$ the gradient of $K$ at $m^0$ and write $P=D_m K(m^0).$ Note that since $K$ is not defined on the whole space $\bR^{n \times n},$ its partial derivatives $D_{m^{ij}} K(m)$ at $m$ do not exist per se. However, we adopt the notation
\[
D_{m^{ij}} K(m):=\big(D_mK(m)\big)^{ij}, \qquad \forall (i, j) \in \bE, \; i<j.
\]
We then have that 
\[
K(m)-K(m^0)-\sum_{i<j} D_{m^{ij}} K\big(m^0\big) \big(m^{ij}-m_0^{ij}\big)=o\big(\|m-m^0\|_{\ell_2}\big).
\]
We denote by ${\bf 1} $ the element of $\bR^n$ whose components are all equal to $1$. The tangent space of $\cP_0(\bG)$ at $\mu$ and the orthogonal projection $\Pi_{\bR^n_0}$ from $\bR^n$ onto the tangent space are respectively denoted by 
$$
\bR^n_0:=\Big\{V \in \bR^n:\; \sum_{i=1}^n V^i=0 \Big\}, \qquad \Pi_{\bR^n_0}: V \to V-{{\bf 1} \over n} \sum_{i=1}^n V^i.
$$ 
Given a function $\cF: \cP_0(\bG) \to \bR$, we denote by $\delta_{\mu} \cF(\mu)$ the Fr\'echet derivative of $\cF$ on $\cP_0(\bG)$ at $\mu$, and thus $\delta_{\mu} \cF(\mu) \in \bR^n_0$. If $\cF$ admits an extension, which we continue to denote by $\cF:[0,1]^n \to \bR$, and the extension has a Fr\'echet derivative $D_{\mu} \cF(\mu)$ on $(0,1)^n$ at $\mu \in (0,1)^n$ then 
\begin{equation}\label{eq:jan22.2026}
\delta_{\mu} \cF=\Pi_{\bR^n_0}\Big( D_\mu \cF\Big). 
\end{equation}
We shall denote the components of $\delta_{\mu} \cF$ by $\delta_{\mu^1} \cF, \cdots, \delta_{\mu^n} \cF$ and the components of $D_\mu \cF$ by $D_{\mu^1} \cF, \cdots, D_{\mu^n} \cF.$

The logarithmic-mean activation functional $\theta$ discussed in the introduction is defined more precisely by
\[
\theta(r, s)= \left\{
     \begin{array}{ll}
       {r-s \over \log r-\log s}& \text{if }  r \neq s, rs>0  \\
        r & \text{if } r =s>0\\
       0  & \text{if }  rs=0
     \end{array}
   \right..
\]
Related proposals for discrete Wasserstein metrics on graphs appear in \cite{EM2, GangboLM, GangboMSwiech, Maas}. Given $\rho = (\rho^1,\dots, \rho^n)\in [0,+\infty)^n$, we set $\theta_{ij}(\rho):=\theta(\rho^i, \rho^j)$. The $\rho$--divergence of $v \in \bS(n)$ is $\nabla_\rho \cdot v$ defined by $(\nabla_\rho \cdot v)^i=\sum_{j \not= i}^n \sqrt{\omega_{ij}} \theta_{ij}(\rho) v^{ji}.$ We shall work with the discrete {\em bilinear form} and {\em semi-norm}  
\begin{equation}\label{eq:inner-product-and-norm}
(v,\tilde v )_ \rho:=\frac{1}{2}\sum_{(i,j)\in \bE} v^{ij}\tilde v^{ij}\theta_{ij}(\rho), \quad \|v\|_\rho=\sqrt{(v,v )_ \rho } \quad \forall \; v, \tilde v \in \bS(n).
\end{equation}
We define
\begin{equation}\label{eq:definitionfij}
f^0(a, b)=\left\{
\begin{array}
[c]{cl}
0, &\;\; \text{if}\;\;a=0, \;\; b=0; \smallskip\\
+\infty, &\;\; \text{if}\;\;a=0, \;\; b\not =0; \smallskip\\
\frac{b^2}{2a}, &\;\; \text{if}\;\;a>0,
\end{array}
\right. 
\end{equation} 
which is useful for introducing the discrete kinetic energy functional  
\[
f^\theta_{ij}(\mu, m)=f^0(\theta_{ij}(\mu), m^{ij}),\quad (\mu, m) \in \cP(\bG) \times \bS(n).
\] 
A class of model Lagrangians / Hamiltonians $\cL_\theta, \cH_\theta: [0,+\infty)^n \times \bS(n) \to [0,+\infty]$ is given by
\begin{equation}\label{eq:jan28.2024.1ter}
\cL_\theta(\mu, m)= \frac{1}{2} \sum_{(i, j) \in \bE}f^\theta_{ij}(\mu, m), \qquad \cH_\theta(\mu, p)=\frac{1}{4} \sum_{(i, j) \in \bE}\theta_{ij}(\mu) (p^{ij})^2.
\end{equation} 
Note that $\cL_\theta$ is convex and lower semicontinuous and $\cH_\theta(\mu, \cdot)$ is the Legendre transform of $\cL_\theta(\mu, \cdot)$. 
When $\mu \in (0,1)^n$ and $p, m \in \bS(n)$,   
\[
D_{m^{ij}}  \cL_\theta(\mu, m)= {m^{ij} \over \theta_{ij}(\mu)}, \qquad D_{p^{ij} } \cH_\theta(\mu, p)=\theta_{ij}p^{ij}
\]
and thus 
\begin{equation}\label{eq:dec21.2025.2}
|D_{p^{ij} } \cH_\theta(\mu, p)| =(\mu^i+\mu^j) h_{\log}\bigg({\mu^j \over \mu^i} \bigg)|p^{ij}|,
\end{equation}
where $ h_{\log} \in C((0,+\infty))$ is defined by $(1+u) |\log u|\; h_{\log}(u)= |u-1|.$

%
%
\subsection{Actions and individual noise operator}\label{subsection:feb14.2026}
Fix a convex Lagrangian $\cL:[0,+\infty)^n\times\bS(n)\to\bR\cup\{+\infty\}$, a running cost $\cF:[0,1]^n\to\bR$, a terminal time $T>0$, and a terminal cost $\cU_T:[0,1]^n\to\bR$; the precise structural assumptions on $(\cL,\cF,\cU_T)$ invoked by Theorem~\ref{thm:main-hjb} are collected in Section~\ref{subsection:assumptions-g}. For curves $\rho:[t,T]\to \mP(\mathbb{G})$ and $m:[t,T]\to \mathbb{S}(n)$, we define the action
\[
\cA_t^T(\rho,  m):= \int_t^T  \Big(\cL(\rho,  m+  \nabla_\bG \rho) -\cF(\rho)\Big)ds.
\]
We aim to minimize $\cA_t^T(\rho,  m)+ \cU_T(\rho(T))$
over the set of $(\rho, m)$ such that $\rho \in C([t, T], \cP(\bG))$, $\rho(t)=\mu$, $m \in L^{p_0}\big((t, T); \bS(n)\big)$ for some $p_0>1$, and 
\begin{equation}\label{eq:may-continuity.ter}
\dot \rho+\nabla_\bG \cdot m=0,
\end{equation} 
in the sense of distributions. 
For such pairs $(\rho,m)$, we write $(\rho, m) \in \cC_t^T(\mu, \cdot)$ and set
\begin{equation}\label{eq:aug01.2025.2}
\cU(t, \mu):= \inf_{\rho, m} \bigg\{ \cA_{t}^{T}(\rho, m) + \cU_T(\rho(T)):\: (\rho, m) \in \cC_t^T(\mu, \cdot) \bigg\}.
\end{equation}

Inspired by the individual noise in mean field games (\cite{ChowGan, GangboM22, GangboMMZ22}), we define the operator $\Delta_\text{ind}$ which, when applied to the Wasserstein gradient $\nabla_{\mW}\mathcal{V} = \nabla_\mathbb{G}\delta_\mu\mathcal{V}$ of a smooth function $\cV: \cP(\bG) \to \bR$, is given by
\begin{equation}\label{eq:sep06.2025.2}
(\Delta_\text{ind} \cV)(\mu)= \big(\nabla_\bG\cdot (\nabla_{\cW}\cV)(\mu), \mu\big)=-\big(\nabla_{\cW}\cV(\mu), \nabla_\bG \mu\big)=-\big(\nabla_{\cW}\cV(\mu), \nabla_\bG \log \mu\big)_\mu.
\end{equation}

%
%
%
%
\section{Interiority estimates  for the continuity equation}\label{sec:interiority}
Throughout this section, $h$, $A$ and the continuity equation~\eqref{eq:conti general} are as in Theorem~\ref{thm:main-interiority}; in particular, $\sum_{i=1}^n\rho_i\equiv1$ since $A$ is skew-symmetric. We prove Theorem~\ref{thm:main-interiority} via two lemmas of independent interest.
\begin{lemma}[Spatial propagation of smallness]\label{lem:interiority 1} For $\epsilon>0$, let $\mu \in \cP_\epsilon(\bG)$, $t_0 \in (0,T]$, and assume that $\rho\in C^1([0,t_0]; (0,1)^n)$ is a classical solution to \eqref{eq:conti general}. There exists a constant $K>1$, depending only on $n$, $\omega_{\min}$, $\omega_{\max}$, and $h$, such that the following holds: if $\delta \in (0,\frac{\epsilon}{K})$, $i_0\in \{1,\ldots,n\}$, and $t_0$ is the first time such that $ \rho_{i_0}(t_0)=\delta,$ 
then
\begin{equation}
    \min_{s\in[0,t_0]}\rho_i(s)\leq K\delta \;\text{ for all }\; i\in\{1,\ldots,n\}.
\end{equation}
\end{lemma}
\begin{proof} Let $C_0>1$ be a constant to be chosen later. Since $\delta<\epsilon/K<\epsilon$, then $t_0>0$ and $\dot \rho_{i_0}(t_0)\leq 0$. Below, we let $C$ denote a generic constant that may increase with each line and depends on $n$, $\omega_{\min}$, $\omega_{\max}$, and $\|h\|_{\infty}$, but does not depend on $C_0$ or $\delta$.  Using \eqref{as:A conti},
\[
 0\geq\dot \rho_{i_0}(t_0)=\sum_{j} \sqrt{\omega_{i_0j}}A_{j i_0}+\sum_j (\rho_j-\rho_{i_0}) \omega_{i_0j} 
 \geq  -CC_0\delta+\sum_{\rho_j>C_0\delta} \sqrt{\omega_{i_0j}}A_{j i_0}+\sum_{\rho_j>C_0\delta} (\rho_j-\rho_{i_0})\omega_{i_0j}
\]
and so, 
\[
0 \geq -CC_0\delta - \sum_{\rho_j>C_0\delta}\sqrt{\omega_{i_0 j}}(\rho_{i_0}+\rho_{j})h(\rho_j/\rho_{i_0}) +\sum_{\rho_j>C_0\delta} (\rho_j-\rho_{i_0})\omega_{i_0j}.
\]
We conclude that  
\[
0 \geq -CC_0\delta - \sum_{\rho_j>C_0\delta}\sqrt{\omega_{i_0 j}}\rho_{j}h(\rho_j/\delta) +\sum_{\rho_j>C_0\delta} \rho_j\omega_{i_0j}=
-CC_0\delta+\sum_{\rho_j>C_0\delta}\sqrt{\omega_{i_0 j}}\rho_j(\sqrt{\omega_{i_0 j}}-h(\rho_j/\delta)).
\]
 Hence, recalling \eqref{eq:h properties}, we may fix $C_0=C_0(h,\omega_{\min})$ large enough so that $\sqrt{\omega_{\min}}-h(u)>\frac{1}{2}\sqrt{\omega_{\min}}$ for $u>C_0$.
We conclude that 
\begin{equation}\label{eq:dec08.2025.1}
   \sum_{\rho_j>C_0\delta} \rho_j\omega_{i_0j}\leq CC_0\delta.
\end{equation}
Let $j$ be a neighbor of $i_0$. Either $\rho_j(t_0)\leq C_0\delta$, and thus $ \min_{[0,t_0]}\rho_j \leq C_0\delta$, or else $\rho_j(t_0)>C_0\delta$, and thus  
\[
 \min_{[0,t_0]}\rho_j \leq \rho_j(t_0) \leq  \sum_{\rho_j>C_0\delta} \rho_j(t_0) \leq {CC_0\delta \over \omega_{\min}}
\]
In conclusion, for a constant $C$ that now depends on $\omega_{\min}$ and for every neighbor $j$ of $i_0$, 
\[
\min_{s\in[0,t_0]}\rho_j(s) \leq \rho_j(t_0)\leq CC_0\delta.
\] 
We now denote by $C_*>1$ the current value of $C$ and fix the choice $K=(C_*C_0)^{n-1}$. Let $i_1$ be a neighbor of $i_0$. Repeating the same argument with $(i_0,\delta,t_0)$ replaced by $(i_1,C_*C_0\delta, \bar{t}_0)$, where $\bar{t}_0\in\operatorname{argmin}_{s\in[0,t_0]}\rho_{i_1}(s)$, we infer that
\[\min_{s\in[0,t_0]}\rho_j(s) \leq \rho_j(\overline{t}_0)\leq(C_*C_0)^2\delta\] for every neighbor $j$ of $i_1$. Proceeding in this way, since $\G$ is connected, we conclude after at most $n-1$ iterations along a shortest path from $i_0$ to any vertex $i$ that
\[    \min_{s\in[0,t_0]}\rho_i(s) \leq (C_*C_0)^{n-1}\delta\leq K\delta, \quad i\in\{1,\ldots,n\}.\]
This concludes the proof. \end{proof}
\begin{lemma}[Waiting time estimate] \label{lem:interiority 2} Let $\mu \in \cP_\epsilon(\bG)$, $0< t_0<t_1\leq T$, and assume that $\rho\in C^1([0,t_1]; (0,1)^n)$ is a classical solution to \eqref{eq:conti general}.  There exist constants $K>1$ and $c_0>0$, depending on $n$, $\omega_{\min}$, $\omega_{\max}$ and $h$, such that the following holds: if $\delta \in (0,\frac{\epsilon}{K})$ and $t_0$, $t_1$ are, respectively, the  first times such that
\begin{equation}
    \min_{i}\rho_i(t_0)=\delta, \;\;\min_{i}\rho_i(t_1)=\frac{1}{2K}\delta, 
\end{equation}
then $t_1-t_0\geq c_0.$
\end{lemma}
\begin{proof}
Let $K$ be the constant of Lemma \ref{lem:interiority 1}. By definition of $t_0$, $\rho_i(s) > \delta$ for all $s\in[0,t_0)$ and all $i \in \{1,\ldots,n\}$. Thus, by Lemma \ref{lem:interiority 1} applied to $i_0\in\operatorname{argmin}_{i}\rho_i(t_1)$ with $\delta$ replaced by $\frac{\delta}{2K}$,
\begin{equation}
\min_{s\in[t_0,t_1]}\rho_i(s)=\min_{s\in[0,t_1]}\rho_i(s)\leq \frac{\delta}{2}\;\;\text{ for every }\; i\in \{1,\ldots,n\}. 
\end{equation}
Since the continuity equation preserves mass, $\sum_{i}\rho_i(t_0)=1$, so $n\delta\leq 1$ and there exists a vertex $i$ such that $\rho_i(t_0)\geq \frac{1}{n}$. On the other hand, since $h$ is bounded, the continuity equation implies that
\begin{equation}
\bigl|\dot{\rho}_i\bigr|
\le \sum_{j}|A_{ij}|\,\sqrt{\omega_{ij}}
    + \sum_{j}|\rho_j-\rho_i|\,\omega_{ij}
\le 2\|h\|_{\infty}\sum_{j}\sqrt{\omega_{ij}}
    + 2\sum_{j}\omega_{ij}
\le C.
\end{equation}
Therefore, letting $t_2\in \operatorname{argmin}_{s\in[t_0,t_1]}\rho_i(s)$ and using the fact that $\delta \leq \frac{1}{n}$, we get
\begin{equation}
  \frac{1}{2n}\leq\frac{1}{n}-\frac{\delta}{2}\leq  \rho_i(t_0)-\rho_i(t_2)\leq C|t_0-t_2|\leq C(t_1-t_0).  
\end{equation}
This proves the claim for $c_0:=1/(2nC)$.
\end{proof}
\begin{proof}[Proof of Theorem~\ref{thm:main-interiority}]
By standard ODE theory, since $\min_{1\leq i\leq n}\mu^i\geq \epsilon>0$ and $A \in C^{1}((0,T)\times (0,1)^n;\bS(n))$, local existence for \eqref{eq:conti general} holds and the unique solution can be continued (with strictly positive components) until a maximal time $T_{\max}\in (0, T]$. Moreover, we must have
\begin{equation}
    T_{\max}=T\;\; \text{ or } \;\liminf_{t\uparrow T_{\max}}\min_{i}\rho_i(t)=0,
\end{equation}
since otherwise the existence interval could be extended. 

Let $\delta_0=\frac{\epsilon}{2K}$, where $K>1$ is the constant of Lemma \ref{lem:interiority 2}. If $\min_{1\leq i\leq n,\; s\in [0,T_{\max})}\rho_i(s) \geq \frac{\delta_0}{2K}$, then $T_{\max}=T$ and there is nothing to prove. Otherwise, by Lemma \ref{lem:interiority 2}, if $t_1$ is the first time such that $\min_{i}\rho(t_1)=\frac{1}{2K}\delta_0$, then
\begin{equation}
    t_1\geq c_0>0.
\end{equation}
If $\min_{i,s}\rho_i(s) \geq \frac{\delta_0}{(2K)^2}$, then again we are done. Otherwise, applying Lemma \ref{lem:interiority 2} again, if $t_2$ is the first time after $t_1$ such that $\min_{i}\rho_i(t_2)=\frac{1}{(2K)^2}\delta_0$, then since $t_2-t_1\ge c_0$ we have 
\begin{equation}
    t_2 \geq 2c_0.
\end{equation}
Continuing in this way, and letting $t_0:=0$, we obtain a sequence $t_0<t_1<\cdots <t_j$ such that
\begin{equation}
    t_j-t_0\geq jc_0, \quad \min_{1\leq i\leq n, s\in [0,t_j]}\rho_i(s)=\min_{1\leq i\leq n}\rho_i(t_j)=\frac{\delta_0}{(2K)^{j}}.
\end{equation}
After finitely many steps, we must have $(j+1)c_0> T$, and therefore, by a final application of Lemma \ref{lem:interiority 2}, we conclude that $T_{\max}=T$, and
\begin{equation}
 \min_{1\leq i\leq n, s\in [0,T]}\rho_i(s)>  \frac{\delta_0}{(2K)^{(j+1)}}=\frac{\epsilon}{(2K)^{(j+2)}}. 
\end{equation}
By the optimality of $j$, one has $jc_0\leq T$, and so
\begin{equation}
 \min_{1\leq i\leq n, s\in [0,T]}\rho_i(s)\geq c\epsilon e^{-rT}, \quad r=\frac{\log(2K)}{c_0}, \quad  c=\frac{1}{(2K)^2}.
\end{equation}
The precise estimate \eqref{eq:rho lower bd} then follows by applying the above argument to the restriction of $\rho$ to the interval $[0,t]$.
\end{proof}

%
%
%
\section{The extended MFG system and the master equation}\label{sec:classical}
Throughout this section, the hypotheses on $H$, $B$, $g$, and the auxiliary function $a$ are those of Theorem~\ref{thm:main-mfg}. For later referencing, we record them here: for all $(\mu,p)\in(0,1)^n\times\bS(n)$ and $(i,j)\in\bE$,
\begin{align}
 (B(\mu,p),p)&\geq (H(\mu,p),\mu)-C_1, \label{as:B coercivity}\\
 H_i(\mu,p)&\geq -C_1, \label{as:H lower bound}\\
 |B_{ij}(\mu,p)|&\leq (\mu^i+\mu^j)\,a(\mu^j/\mu^i,p), \label{eq:mobility-dominated flux}\\
 \max_{1\leq i\leq n,\,\rho\in[0,1]^n}|g_i(\rho)|&\leq C_1. \label{as:g mu}
\end{align}
The uniqueness half of Theorem~\ref{thm:main-mfg} additionally requires the Lasry--Lions monotonicity conditions: for all $\mu_1,\mu_2\in\cP_0(\G)$ and $p_1,p_2\in\bS(n)$ with $(\mu_1,p_1)\neq(\mu_2,p_2)$,
\begin{align}
& (\mu_1-\mu_2,H(\mu_1,p_1)-H(\mu_2,p_2))< (p_1-p_2,B(\mu_1,p_1)-B(\mu_2,p_2)),\label{as:uniqueness 1}\\
& (g(\mu_1)-g(\mu_2),\mu_1-\mu_2)\geq 0.\label{as:uniqueness 2}
\end{align}
Well-posedness for extended MFG on a continuous state space under~\eqref{as:uniqueness 1} has been studied, for instance, in \cite{LiSo20,Mun23}. When treating the master equation, we will use the following slightly stronger (differential) version of~\eqref{as:uniqueness 1}: for all $(\rho,p)\in(0,1)^n\times\bS(n)$ and $(\eta,q)\in(\bR^n_0\times\bS(n))\setminus\{(0,0)\}$,
\begin{equation}\label{as:uniqueness 3}
(D_\mu H(\rho,p)[\eta] + D_p H(\rho,p)[q], \eta) < (D_\mu B(\rho,p)[\eta] + D_p B(\rho,p)[q], q).
\end{equation}

This section proves Theorem~\ref{thm:main-mfg} in two steps. First, a priori bounds for a homotopic family of MFG systems, combined with Theorem~\ref{thm:main-interiority}, yield existence and uniqueness of classical solutions to~\eqref{eq:mfg} via a Schaefer fixed-point argument. Second, continuous differentiability of the solution map in the initial data allows us to define the value function and verify that it is the unique classical solution of the master equation~\eqref{eq:master}.

\subsection{Existence and uniqueness of classical solutions to the MFG system}
We begin by obtaining a priori estimates, uniform in $\lambda\in[0,1]$ and $(t,\mu)\in[0,T)\times\cP_0(\G)$, for classical solutions $(\phi,\rho):[t,T]\to\R^n\times(0,1)^n$ of the homotopic family
\begin{equation}  \label{eq:mfg lambda}
    \begin{cases} \dot \phi(s)= H(\rho(s),\lambda\nabla_{\G}\phi(s))-\Delta_{\G}\phi(s)& s\in (t,T)\\
    \dot \rho(s)=\nabla_{\G}\cdot (B(\rho(s),\lambda\nabla_{\G}\phi(s)))+\Delta_{\G}\rho(s)& s\in (t,T)\\
    \phi(T)= g(\rho(T)), \quad \rho(t)=\mu.
    \end{cases}
\end{equation}
\begin{lemma}[Lasry--Lions estimate] \label{lem:lasrylions} Let $(\phi,\rho)$ be a classical solution to \eqref{eq:mfg lambda}. Then
\begin{equation}
\lambda ( \phi(T),\rho(T))\leq \lambda ( \phi(t),\mu)+2C_1(T-t).
\end{equation}
\end{lemma}
\begin{proof}
By \eqref{as:B coercivity}, \eqref{as:H lower bound}, and the duality of \eqref{eq:mfg lambda}, we have
\begin{align*}
    \frac{d}{ds}(\lambda\phi(s),\rho(s))=&- (B(\rho(s),\lambda \nabla_{\G}\phi(s)), \lambda\nabla_{\G}\phi(s)) + \lambda (H(\rho(s),\lambda \nabla_{\G}\phi(s)) ,\rho)  \\
 \leq & C_1- (1-\lambda) (H(\rho,\lambda \nabla_{\G}\phi), \rho)\leq 2 C_1.
\end{align*}

We conclude the proof by integrating from $t$ to $T$.
\end{proof}
\begin{lemma}[Upper bound on $\phi$] \label{lem:upper phi} If $(\phi,\rho)$ is a classical solution to \eqref{eq:mfg lambda}, then
    \begin{equation}
        \max_{1\leq i \leq n} \phi_i(s) \leq  C_1(T-s)+\max_{1\leq i\leq n}  \phi_i(T), \quad s\in [t, T].
    \end{equation}
\end{lemma}
\begin{proof}
For $s\in [t,T]$, let $M(s)=\max_{i}\phi_i(s)$. Since $\phi$ is a classical solution, we see that $M$ is Lipschitz continuous and, if $i_s\in \operatorname{argmax}_{1\leq j\leq n}(\phi_j(s))$, then $M'(s)=\dot \phi_{i_s}(s)$ for almost every $s$. Moreover,  $\Delta_{\G}\phi_{i_s}(s)\leq0$, and thus from \eqref{as:H lower bound},
\begin{equation}
    M'(s)=\dot \phi_{i_s}(s)\geq H_{i_s}\geq -C_1.
\end{equation}
The claim follows by integrating over $[s,T]$.
\end{proof}
\begin{proposition}[Uniform bounds on $\lambda \phi$]\label{prop:phi bd} Let $\mu \in \cP_\epsilon(\G)$ and $t\in[0,T)$. For any classical solution $(\phi,\rho)$ to \eqref{eq:mfg lambda},
\begin{equation}
\epsilon \big\|\max_{1\leq i\leq n}| \lambda\phi_i|\big\|_{L^{\infty}(t,T)}\leq (4(T-t)+3)C_1.
\end{equation}
\end{proposition}
\begin{proof} By Lemma \ref{lem:upper phi}, we have
\begin{equation} \label{eq:lphi upper pf}
\max_{1\leq i\leq n} \lambda \phi_i(s) \leq \lambda (T-t+1)C_1, \quad s\in [t,T],
\end{equation}
so to prove the result it suffices to bound from below the function $s \mapsto \lambda\Phi(s)=(\lambda \phi(s),\textbf{1})=\sum_{i=1}^{n}\lambda\phi_i(s).$
Testing the first equation of \eqref{eq:mfg lambda} against \textbf{1}, using \eqref{as:H lower bound} and the fact that $(\Delta_{\G}\phi,\textbf{1})=0$, we have
$\dot \Phi(s)\geq\sum_{i}H_i \geq -C_1n.$
Therefore, noting that $1=\sum_{i}\mu^i\geq n\epsilon$, we obtain that
\begin{equation} \label{eq:dec16.2025.1}
\lambda \Phi(s)\geq \lambda \Phi(t)-\lambda C_1n(s-t)\geq \lambda \Phi(t)-\frac{C_1 }{\epsilon}(s-t),\quad s\in [t,T],
\end{equation}
so it suffices to bound $\lambda \Phi(t)$ from below. By Lemma \ref{lem:lasrylions}, we have
\begin{equation}
    -C_1 \leq (\lambda \phi(t),\mu)+2C_1(T-t) = (\lambda \phi(t)^+,\mu)+2C_1(T-t) -(\lambda\phi(t)^-,\mu).
\end{equation}
Moreover, by Lemma \ref{lem:upper phi} and \eqref{as:g mu},
\[
(\lambda \phi(t)^+,\mu)\le \max_{i}\lambda \phi_i(t)\le \max_i \phi_i(t)\le C_1(T-t)+\max_i\phi_i(T)\le C_1(T-t+1),
\]
and therefore
\begin{equation}
    -C_1 \leq (\lambda \phi(t),\mu)+2C_1(T-t) \leq C_1(2(T-t)+1) -(\lambda \phi(t)^-,\mu).
\end{equation}
Thus,
\begin{equation}\label{eq:dec16.2025.2}
    \lambda \Phi(t)^-\leq (\lambda \phi^-(t),\textbf{1})\leq \frac{1}{\epsilon}(\lambda \phi^-(t),\mu)\leq\frac{3C_1(T-t+1)}{\epsilon}.
\end{equation}
Recalling \eqref{eq:lphi upper pf}, we conclude that
\[
\lambda \phi_i(s)=  \lambda\Phi(s)-\sum_{k \not =i } \lambda \phi_k(s)
\geq \lambda\Phi(s) - (n-1)\max_k\lambda\phi_k(s),
\]
for all $i$. We use \eqref{eq:dec16.2025.1} and \eqref{eq:dec16.2025.2} to get that
\[
\lambda\Phi(s)\ge -\frac{3C_1(T-t+1)}{\epsilon}-\frac{C_1}{\epsilon}(s-t)\ge -\frac{C_1(4(T-t)+3)}{\epsilon}, \quad s\in[t,T],
\]
and hence, using again that $n\epsilon\le 1$,
\[
\lambda \phi_i(s)\ge -\frac{C_1(4(T-t)+3)}{\epsilon}-(n-1)C_1(T-t+1)\ge -\frac{C_1(5(T-t)+4)}{\epsilon},
\]
for all $i$ and $s \in [t,T].$ Together with \eqref{eq:lphi upper pf}, this proves the claim.
\end{proof}


We may now combine the estimates on $\phi$ with the interiority estimates of the previous section to obtain a lower bound on $\rho_i$ that does not deteriorate for arbitrarily small time horizons.
\begin{lemma}[A priori lower bound on $\rho$]\label{lem:rho lower bound mfg lambda}
Let $\mu \in \cP_\epsilon(\G)$, $\lambda\in[0,1]$, $t\in[0,T)$, and let $(\phi,\rho)$ be a classical solution to \eqref{eq:mfg lambda}. There exists a constant $c>0$, depending only on $n$, $\omega_{\min}$, $\omega_{\max}$, $a$, $C_1$, $\epsilon$, and monotonically on $T$, such that
\begin{equation}\label{eq:rho lower bd mfg lambda}
    \min_{1\leq i\leq n} \rho_i(s)\geq c\epsilon, \quad s\in[t,T].
\end{equation}
\end{lemma}
\begin{proof}
Let $L:=\left\|\max_{1\leq i\leq n}|\lambda\phi_i|\right\|_{L^{\infty}(t,T)}$. By Proposition \ref{prop:phi bd}, $L\leq \frac{(5(T-t)+4)C_1}{\epsilon}\leq \frac{(5T+4)C_1}{\epsilon}$. Moreover, for $(i,j)\in \bE$,
\[
|(\lambda\nabla_{\G}\phi)_{ij}|=\sqrt{\omega_{ij}}\,|\lambda\phi_j-\lambda\phi_i|
\leq 2\sqrt{\omega_{\max}}\,L=:M.
\]
The constant $M = M(T)$ depends only on $n$, $\omega_{\max}$, $C_1$, $\epsilon$, and monotonically on $T$, and bounds
\[
\|\lambda\nabla_{\G}\phi\|_{L^{\infty}(t,T)}\leq M.
\]
Define
\[
h(u):=\sup_{|p|\leq M}a(u,p), \quad u>0.
\]
Then $h$ is bounded and satisfies \eqref{eq:h properties}. In addition, by \eqref{eq:mobility-dominated flux},
\[
|B_{ij}(\rho(s),\lambda\nabla_{\G}\phi(s))|
\leq (\rho_i(s)+\rho_j(s))\,h(\rho_j(s)/\rho_i(s)), \quad (i,j)\in \bE,
\]
so the second equation in \eqref{eq:mfg lambda} is of the form \eqref{eq:conti general}. Applying Theorem~\ref{thm:main-interiority} to the time-translated equation $\tilde\rho(\sigma):=\rho(\sigma+t)$ on $[0,T-t]$ yields $\min_{i}\rho_i(s)\geq c'\epsilon e^{-r'(s-t)}\geq c'\epsilon e^{-r'T}$ for $s\in[t,T]$, with $c',r'>0$ depending only on the above parameters; setting $c:=c'e^{-r'T}$ gives \eqref{eq:rho lower bd mfg lambda}. Finally, since the above choice of $M$ is monotone in $T$, the constant $c$ may be chosen monotonically in $T$ as well (in particular, it does not deteriorate as $T\downarrow 0$).
\end{proof}
%

With these estimates in hand, we may now establish the main result of this subsection.
\begin{proposition}[Existence and uniqueness for the extended MFG system]\label{prop:mfg-classical}
Under the hypotheses of Theorem~\ref{thm:main-mfg}, for every $(t,\mu)\in[0,T)\times\cP_0(\G)$, the system~\eqref{eq:mfg} admits a classical solution $(\phi,\rho)\in C^2([t,T];\R^n\times(0,1)^n)$, which is unique if additionally~\eqref{as:uniqueness 1}--\eqref{as:uniqueness 2} hold.
\end{proposition}
\begin{proof}
The system~\eqref{eq:mfg} corresponds to the case $\lambda=1$ of the homotopic family~\eqref{eq:mfg lambda}.
Fix $(t,\mu)\in[0,T)\times\cP_0(\G)$. Let $\mathbb X=C^0([t, T] ; \R^n)$ be equipped with the uniform norm. Noting that $\mathbb X$ is a Banach space, we define an operator $\mathcal T : \mathbb X \to \mathbb X$ as follows. Let $\varphi \in \mathbb X$. By Theorem~\ref{thm:main-interiority}, applied to the time-translated equation $\tilde\rho(\sigma):=\rho(\sigma+t)$ on $[0,T-t]$ with $A_{ij}(\sigma,\mu):=B_{ij}(\mu,\nabla_{\G}\varphi(\sigma+t))$ and $h(u)=\sup_{|p|\leq \|\nabla_{\G}\varphi\|_{\infty}}a(u,p)$, there exists a unique classical solution $\rho \in C^1([t, T];(0,1)^n)$ to
\begin{equation}\label{eq:conti fix pf}
    \dot \rho(s)= \nabla_{\G}\cdot(B(\rho(s),\nabla_{\G}\varphi(s)))+\Delta_{\G}\rho(s), \quad s\in [t,T],\quad \rho(t)=\mu.
\end{equation}
Since $\rho_i>0$ for all $i$, the backward ODE
\begin{equation}
    \dot \phi(s)=H(\rho(s),\nabla_{\G}\varphi(s))-\Delta_{\G}\phi(s), \;\; s\in [t,T], \quad \phi(T)=g(\rho(T)),
\end{equation}
then has a unique classical solution $\phi$, and we may define $\mathcal T(\varphi):=\phi$. Thanks to the lower bound \eqref{eq:rho lower bd}, smooth dependence on parameters for ODE solutions, and the Arzel\`a-Ascoli theorem, the operator $\mathcal T$ is continuous and compact. More precisely, given a bounded family $\{\varphi_{\alpha}\}_{\alpha}\subset \mathbb X$ with $\sup_{\alpha}\|\nabla_{\G}\varphi_{\alpha}\|_{L^{\infty}(t,T)}\leq M$, we may apply Theorem~\ref{thm:main-interiority} with the fixed choice $h(u):=\sup_{\|p\|\leq M}a(u,p)$.   Moreover, by Proposition \ref{prop:phi bd}, the set
\begin{equation}
    \{\varphi\in \mathbb X: \varphi=\lambda \mathcal T(\varphi) \;\;\text{for some } \;0\leq\lambda\leq 1\}
\end{equation}
is bounded. Therefore, by Schaefer's fixed point theorem, $\mathcal T$ has a fixed point $\varphi$, so that $\phi:=\mathcal T(\varphi)=\varphi$. Defining $\rho$ by \eqref{eq:conti fix pf}, we see that $(\phi,\rho)$ is a classical solution to \eqref{eq:mfg}, and $C^{2}$ regularity follows from $H\in C^1$ and $B\in C^1$. Finally, uniqueness under \eqref{as:uniqueness 1} and \eqref{as:uniqueness 2} follows by adapting the classical Lasry--Lions argument \cite{LL2} to our discrete setting. If $(\rho^i, \phi^i)$, $i=1,2$, are two solutions to \eqref{eq:mfg}, differentiating $(\rho^1-\rho^2, \phi^1-\phi^2)$ and using \eqref{eq:mfg} gives
\begin{align*}
& \Big(\rho^1(T)-\rho^2(T), \phi^1(T)-\phi^2(T) \Big)- \Big(\rho^1(t)-\rho^2(t), \phi^1(t)-\phi^2(t) \Big)\\
=& \int_t^T \Big( \big(\rho^1-\rho^2, H(\rho^1, \nabla_{\G} \phi^1)-H(\rho^2, \nabla_{\G} \phi^2)\big)-
\big(B(\rho^1, \nabla_{\G} \phi^1)-B(\rho^2, \nabla_{\G} \phi^2),  \nabla_{\G} \phi^1- \nabla_{\G} \phi^2\big)
\Big)ds.
\end{align*}
Since $\rho^1(t)=\rho^2(t)=\mu$, \eqref{as:uniqueness 2} makes the left-hand side nonnegative, while if $(\rho^1,\nabla_\G\phi^1)\not=(\rho^2,\nabla_\G\phi^2)$ on a set of positive measure, \eqref{as:uniqueness 1} forces the right-hand side to be strictly negative. Hence $(\rho^1, \nabla_{\G} \phi^1)=(\rho^2, \nabla_{\G} \phi^2)$ a.e., and by continuity $\rho^1=\rho^2$ and $\nabla_{\G} \phi^1=\nabla_{\G} \phi^2$ on $[t,T]$. Since $\bG$ is connected, $\phi^1-\phi^2$ is spatially constant, i.e.\ $\phi^1_i-\phi^2_i=\kappa(s)$ for some $\kappa\in C^2([t,T])$; the first equation in \eqref{eq:mfg} forces $\kappa$ to be constant, and $\phi^1(T)=\phi^2(T)$ then yields $\kappa\equiv 0$.
\end{proof}

\subsection{Continuous differentiability with respect to the initial time and measure}
Having established existence and uniqueness of classical solutions $(\phi^{t,\mu},\rho^{t,\mu})$ to~\eqref{eq:mfg} for every $(t,\mu)\in[0,T)\times\cP_0(\G)$ in the preceding subsection, we now turn to the regularity of the solution map with respect to the initial data $(t,\mu)$.

\begin{proposition}[$C^1$-dependence on initial data]\label{prop:C1-dependence}
Assume \eqref{as:B coercivity}--\eqref{as:uniqueness 3}, and $g \in C^1((0,1)^n;\R^n)$. Let $(\phi^{t,\mu}, \rho^{t,\mu})$ denote the unique solution of \eqref{eq:mfg} on $[t,T]$ with initial measure $\mu \in \cP_0(\G)$. The rescaled solution map
\[
\cS : [0,T] \times \cP_0(\G) \to C^1([0,1]; \R^n)^2, \qquad \cS(t,\mu)(\theta) = \big(\phi^{t,\mu}(t + (T-t)\theta), \, \rho^{t,\mu}(t + (T-t)\theta)\big)
\]
(with $\cS(T,\mu)(\theta) := (g(\mu), \mu)$) is of class $C^1$ on $[0,T] \times \cP_0(\G)$.
For $t \in [0,T)$ and $\nu \in \R^n_0$, if we define
 \[
 (\psi, \eta)(s) := (\delta_\mu \cS(t,\mu)[\nu])\Big({s-t \over T-t}\Big), \qquad \forall s \in [t,T],
 \]
so that $(\psi, \eta) = \big(\delta_\mu\phi^{t,\mu}[\nu], \delta_\mu \rho^{t,\mu}[\nu] \big)$, then $(\psi, \eta)$ solves the linearized system:
\begin{equation} \label{eq:lin-mfg}
\begin{cases}
\dot \psi(s) = D_\mu H(\rho^{t,\mu}(s), \nabla_\G \phi^{t,\mu}(s))[\eta(s)] + D_p H(\rho^{t,\mu}(s), \nabla_\G \phi^{t,\mu}(s))[\nabla_\G \psi(s)] - \Delta_\G \psi(s), \\[1mm]
\dot \eta(s) = \nabla_\G \cdot \big( D_\mu B(\rho^{t,\mu}(s), \nabla_\G \phi^{t,\mu}(s))[\eta(s)] + D_p B(\rho^{t,\mu}(s), \nabla_\G \phi^{t,\mu}(s))[\nabla_\G \psi(s)] \big) + \Delta_\G \eta(s), \\[1mm]
\psi(T) = Dg(\rho^{t,\mu}(T))[\eta(T)], \qquad \eta(t) = \nu.
\end{cases}
\end{equation}
\end{proposition}

\begin{remark}
In Proposition~\ref{prop:C1-dependence}, we view $\cP_0(\G)$ as an open subset of the affine hyperplane $\big\{\mu \in \bR^n:\; \sum_i \mu^i = 1\big\}$, so $\delta_\mu \cS(t,\mu)$ acts on tangent vectors $\nu \in \R^n_0$; equivalently, one can work with an extension of $\cS$ to $\R \times (0,1)^n$ having the same continuity and regularity properties (this is done in the proof below). In particular, since $D_\mu \cS(t,\mu)[\nu]$ exists, so does $\delta_\mu \cS(t,\mu)[\nu]$, and~\eqref{eq:jan22.2026} relates the full gradient to the constrained one.
\end{remark}

\begin{proof}
Fix $(t_0,\mu_0) \in [0,T) \times \cP_0(\G)$. By Proposition~\ref{prop:mfg-classical}, the unique MFG solution $(\phi^{t_0,\mu_0}, \rho^{t_0,\mu_0})$ on $[t_0,T]$ satisfies $\rho^{t_0,\mu_0} \in C([t_0,T]; (0,1)^n)$. Hence the trajectory $\big(\rho^{t_0,\mu_0}(s), \nabla_\G \phi^{t_0,\mu_0}(s)\big)_{s \in [t_0,T]}$ lies in a compact subset $K \subset (0,1)^n \times \bS(n)$. The set $K \times \{(\eta, q) \in \bR^n_0 \times \bS(n) : |\eta|^2 + |q|^2 = 1\}$ is compact, and condition \eqref{as:uniqueness 3} holds with strict inequality there, so by homogeneity there exists $\alpha > 0$ such that for all $(\varrho, p) \in K$ and $(\eta, q) \in \bR^n_0 \times \bS(n)$, 
\begin{equation}\label{eq:strong monotonicity quant}
(D_\mu H(\varrho,p)[\eta] + D_p H(\varrho,p)[q], \eta) - (D_\mu B(\varrho,p)[\eta] + D_p B(\varrho,p)[q], q) \le -\alpha (|\eta|^2 + |q|^2).
\end{equation}
Differentiating the monotonicity condition \eqref{as:uniqueness 2} shows that for all $\varrho \in \cP_0(\G)$ and $\eta \in \R^n_0$,
\begin{equation}\label{eq:Dg monotonicity}
(\eta, Dg(\varrho)\eta) \ge 0.
\end{equation} 
Let $\mathbb X := C^1([0,1]; \R^n)^2$ and $\mathbb Y := C^0([0,1]; \R^n)^2 \times \R^n \times \R^n$. Define
\[
\mathbb X_* := \big\{(\tilde\phi, \tilde\rho) \in \mathbb X :\; \tilde\rho([0,1]) \subset (0,1)^n \big\},
\]
which is an open subset of $\mathbb X$. Let $\tau := T - t \in (0,T]$. For each $(t,\mu)$ near $(t_0,\mu_0)$, define rescaled time $\theta \in [0,1]$ via $s = t + \tau\theta$, and define rescaled functions
\[
\Big(\tilde\phi^{\tau,\mu}, \tilde\rho^{\tau,\mu}\Big)(\theta) := \big(\phi^{t,\mu},  \rho^{t,\mu}\big)(t + \tau\theta) \equiv \cS(T-\tau, \mu), 
\]
They satisfy the fixed-interval boundary value problem on $[0,1]$:
\begin{equation}\label{eq:rescaled-mfg-C1}
\begin{cases}
(\tilde\phi)' - \tau\big(H(\tilde\rho, \nabla_\G\tilde\phi) - \Delta_\G\tilde\phi\big) = 0, \\[1mm]
(\tilde\rho)' - \tau\big(\nabla_\G \cdot B(\tilde\rho, \nabla_\G\tilde\phi) + \Delta_\G\tilde\rho\big) = 0, \\[1mm]
\tilde\phi(1) = g(\tilde\rho(1)), \qquad \tilde\rho(0) = \mu,
\end{cases}
\end{equation}
where prime denotes $d/d\theta$.  Define the nonlinear map $\Phi : \bR \times (0,1)^n \times \mathbb X_* \to \mathbb Y$ by
\[
\Phi(\tau, \mu, \tilde\phi, \tilde\rho) := \Big( (\tilde\phi)' - \tau(H(\tilde\rho, \nabla_\G\tilde\phi) - \Delta_\G\tilde\phi), \; (\tilde\rho)' - \tau(\nabla_\G \cdot B(\tilde\rho, \nabla_\G\tilde\phi) + \Delta_\G\tilde\rho), \; \tilde\phi(1) - g(\tilde\rho(1)), \; \tilde\rho(0) - \mu \Big).
\]
We let $\tau$ range over $\bR$ instead of $[0,T]$ to ensure that $\tau=0$ and $\tau=T$ are interior points of the domain, for a later application of the implicit function theorem.
Then $(\tilde\phi, \tilde\rho)$ solves \eqref{eq:rescaled-mfg-C1} if and only if $\Phi(\tau, \mu, \tilde\phi, \tilde\rho) = 0$. Since $\nabla_\G$ and $\Delta_\G$ are bounded linear operators on $\R^n$, and $H, B, g$ are $C^1$, the map $\Phi$ is $C^1$. Let $\tau_0 := T - t_0$ and write $(\tilde\phi_0, \tilde\rho_0) := (\tilde\phi^{\tau_0,\mu_0}, \tilde\rho^{\tau_0,\mu_0})$ for brevity. Then $\Phi(\tau_0, \mu_0, \tilde\phi_0, \tilde\rho_0) = 0$.

We next establish invertibility of the partial derivative. Consider the linear operator
\[
L_0 := D_{(\tilde\phi, \tilde\rho)} \Phi(\tau_0, \mu_0, \tilde\phi_0, \tilde\rho_0) : \mathbb X \to \mathbb Y.
\]
Writing $\tilde p_0(\theta) := \nabla_\G\tilde\phi_0(\theta)$, for $(\tilde\psi, \tilde\eta) \in \mathbb X$ one computes
\begin{align*}
L_0(\tilde\psi, \tilde\eta) = \Big( & \tilde\psi' - \tau_0\big(D_\mu H(\tilde\rho_0, \tilde p_0)[\tilde\eta] + D_p H(\tilde\rho_0, \tilde p_0)[\nabla_\G\tilde\psi] - \Delta_\G\tilde\psi\big), \\
& \tilde\eta' - \tau_0\big(\nabla_\G \cdot (D_\mu B(\tilde\rho_0, \tilde p_0)[\tilde\eta] + D_p B(\tilde\rho_0, \tilde p_0)[\nabla_\G\tilde\psi]) + \Delta_\G\tilde\eta\big), \\
& \tilde\psi(1) - Dg(\tilde\rho_0(1))[\tilde\eta(1)], \; \tilde\eta(0) \Big).
\end{align*}
In particular, if $(\tilde\psi, \tilde\eta)$ solves $L_0(\tilde\psi, \tilde\eta) = (0, 0, 0, \nu)$ and we define $(\psi, \eta)(s) := (\tilde\psi, \tilde\eta)\big((s - t_0)/\tau_0\big)$ for $s \in [t_0, T]$, then $(\psi, \eta)$ solves the linearized system \eqref{eq:lin-mfg} with $(t,\mu) = (t_0,\mu_0)$.

To establish injectivity of $L_0$, suppose $L_0(\tilde\psi, \tilde\eta) = (0, 0, 0, 0)$. A direct computation gives
\[
\frac{d}{d\theta}(\tilde\eta, \tilde\psi) = (\tilde\eta', \tilde\psi) + (\tilde\eta, \tilde\psi')
\]
and using discrete integration by parts yields
\[
\frac{d}{d\theta}(\tilde\eta, \tilde\psi) = \tau_0 \Big[ (D_\mu H(\tilde\rho_0, \tilde p_0)[\tilde\eta] + D_p H(\tilde\rho_0, \tilde p_0)[\nabla_\G \tilde\psi], \tilde\eta) - (D_\mu B(\tilde\rho_0, \tilde p_0)[\tilde\eta] + D_p B(\tilde\rho_0, \tilde p_0)[\nabla_\G \tilde\psi], \nabla_\G \tilde\psi) \Big].
\]
Since $\tilde\eta(0) = 0$ and the right-hand side of the forward equation in \eqref{eq:lin-mfg} lies in $\bR^n_0$, we have $\tilde\eta(\theta) \in \bR^n_0$ for all $\theta$. By condition \eqref{eq:strong monotonicity quant}, and since $\tau_0 > 0$, this gives $\frac{d}{d\theta}(\tilde\eta, \tilde\psi) \le -\tau_0\alpha ( |\tilde\eta(\theta)|^2 + |\nabla_\G \tilde\psi(\theta)|^2 )$. Integrating from $0$ to $1$ and using $\tilde\eta(0) = 0$ yields
\[
(\tilde\eta(1), \tilde\psi(1)) \le -\tau_0\alpha \int_0^1 \big( |\tilde\eta|^2 + |\nabla_\G \tilde\psi|^2 \big) \, d\theta.
\]
But $\tilde\psi(1) = Dg(\tilde\rho_0(1))[\tilde\eta(1)]$, so by \eqref{eq:Dg monotonicity}, $(\tilde\eta(1), \tilde\psi(1)) \ge 0$. Therefore $\tilde\eta \equiv 0$ and $\nabla_\G \tilde\psi \equiv 0$. Since the graph $(\bV, \G, \omega)$ is connected, $\nabla_\G \tilde\psi \equiv 0$ implies $\tilde\psi_i(\theta) = \tilde\psi_j(\theta)$ for all $i, j \in \bV$ and $\theta \in [0,1]$. Substituting back into the homogeneous backward equation and recalling that $\tilde\eta \equiv 0$ yields $\tilde\psi' = 0$, so $\tilde\psi$ is constant in $\theta$ as well. Finally, $\tilde\psi(1) = Dg(\tilde\rho_0(1))[\tilde\eta(1)] = 0$ implies $\tilde\psi \equiv 0$, establishing that $\ker L_0 = \{0\}$.

For surjectivity, let $(f_1, f_2, \xi, \nu) \in \mathbb Y$ be given. We aim to find $(\tilde\psi,\tilde\eta) \in \mathbb{X}$ such that $L_0(\tilde\psi, \tilde\eta) = (f_1, f_2, \xi, \nu)$, i.e.,
\begin{equation}\label{eq:shooting-system}
\begin{cases}
\tilde\psi' = \tau_0\big(D_\mu H(\tilde\rho_0, \tilde p_0)[\tilde\eta] + D_p H(\tilde\rho_0, \tilde p_0)[\nabla_\G \tilde\psi] - \Delta_\G \tilde\psi\big) + f_1(\theta), \\
\tilde\eta' = \tau_0\big(\nabla_\G \cdot (D_\mu B(\tilde\rho_0, \tilde p_0)[\tilde\eta] + D_p B(\tilde\rho_0, \tilde p_0)[\nabla_\G \tilde\psi]) + \Delta_\G \tilde\eta\big) + f_2(\theta), \\
\tilde\psi(1) = Dg(\tilde\rho_0(1))[\tilde\eta(1)] + \xi, \qquad \tilde\eta(0) = \nu.
\end{cases}
\end{equation}
We use a shooting argument. For each $q \in \R^n$, let $(\tilde\psi^q, \tilde\eta^q)$ denote the unique $C^1$ solution of the first two equations in \eqref{eq:shooting-system} with initial data $\tilde\psi^q(0) = q$ and $\tilde\eta^q(0) = \nu$; existence and uniqueness follow from continuity of the coefficients in $\theta$. Define the shooting map
\[
\mathcal{M}(q) := \tilde\psi^q(1) - Dg(\tilde\rho_0(1))[\tilde\eta^q(1)] - \xi \in \R^n.
\]
Since the ODE is linear in $(\tilde\psi, \tilde\eta)$, the map $\mathcal{M}$ is affine in $q$. Its linear part $\mathcal{M}_0$ is obtained by setting $f_1 = f_2 = \nu = 0$: for each $q \in \R^n$, let $(\tilde\psi^q_0, \tilde\eta^q_0)$ solve the homogeneous system with $\tilde\eta^q_0(0) = 0$ and $\tilde\psi^q_0(0) = q$, and set $\mathcal{M}_0(q) := \tilde\psi^q_0(1) - Dg(\tilde\rho_0(1))[\tilde\eta^q_0(1)]$. If $\mathcal{M}_0(q) = 0$, then $(\tilde\psi^q_0, \tilde\eta^q_0) \in \ker L_0$, so by injectivity of $L_0$ we have $q = 0$. Thus $\mathcal{M}_0 : \R^n \to \R^n$ is injective, hence bijective, and the affine equation $\mathcal{M}(q) = 0$ has a unique solution $q^*$. The corresponding $(\tilde\psi^{q^*}, \tilde\eta^{q^*})$ is the unique solution of \eqref{eq:shooting-system}, establishing surjectivity. By the open mapping theorem, $L_0$ is an isomorphism.

By the implicit function theorem \cite[Thm.~15.1, Cor.~15.1]{Deimling}, there exists an open neighborhood $\cO \subset \bR \times (0,1)^n$ of $(\tau_0, \mu_0)$ and a unique $C^1$ map 
\[
(\tau, \mu) \mapsto (\tilde\phi^{\tau,\mu}, \tilde\rho^{\tau,\mu}) \in \mathbb X
\]
such that $\Phi(\tau, \mu, \tilde\phi^{\tau,\mu}, \tilde\rho^{\tau,\mu}) = 0$ for all $(\tau, \mu) \in \cO$. Undoing the change of variables $\tau = T - t$, we find that the map $(t,\mu) \mapsto (\phi^{t,\mu}, \rho^{t,\mu})$ is $C^1$ as a map into $C^1$ trajectories (locally around every $(t_0, \mu_0)$), and in particular, the ``initial costate'' $u(t,\mu) := \phi^{t,\mu}(t) = \tilde\phi^{\tau,\mu}(0)$ is a $C^1$ function of $(t,\mu)$. Restricting to $\big((0,T] \times \cP_0(\G)\big) \cap \cO$, we obtain a $C^1$ branch of solutions parametrized by initial time and measure. By uniqueness of solutions, this branch coincides with the MFG solution map. Since $(t_0, \mu_0) \in [0,T) \times \cP_0(\G)$ was arbitrary, $\cS$ is $C^1$ on $[0,T) \times \cP_0(\G)$.

Differentiating the identity $\Phi(\tau, \mu, \tilde\phi^{\tau,\mu}, \tilde\rho^{\tau,\mu}) = 0$ in the direction $\nu$ (with $\tau$ fixed) gives
\[
0 = D_\mu \Phi[\nu] + L_{\tau,\mu} \big( D_\mu(\tilde\phi^{\tau,\mu}, \tilde\rho^{\tau,\mu})[\nu] \big),
\]
where $L_{\tau,\mu} := D_{(\tilde\phi, \tilde\rho)} \Phi(\tau, \mu, \tilde\phi^{\tau,\mu}, \tilde\rho^{\tau,\mu})$.
Since $\Phi$ depends on $\mu$ only through the last component $\tilde\rho(0) - \mu$, we have $D_\mu \Phi[\nu] = (0, 0, 0, -\nu)$. Therefore $(\tilde\psi, \tilde\eta) := D_\mu\cS(t,\mu)[\nu]$ solves $L_{\tau,\mu}(\tilde\psi, \tilde\eta) = (0, 0, 0, \nu)$. Defining $(\psi, \eta)(s) := (\tilde\psi, \tilde\eta)\big((s-t)/(T-t)\big)$ for $s \in [t,T]$, we obtain that $(\psi, \eta)$ solves the linearized system \eqref{eq:lin-mfg}.

Finally, we extend the regularity up to $t=T$ (i.e., $\tau=0$). Fix $\mu_0 \in \cP_0(\G)$. When $\tau=0$, the rescaled system \eqref{eq:rescaled-mfg-C1} reduces to $(\tilde\phi)'=0$, $(\tilde\rho)'=0$ with $\tilde\phi(1)=g(\tilde\rho(1))$, $\tilde\rho(0)=\mu$; the unique solution is $\tilde\phi \equiv g(\mu_0)$, $\tilde\rho \equiv \mu_0$, which gives $\cS(T,\mu_0)(\theta)=(g(\mu_0),\mu_0)$. Set $\tilde\phi_0 \equiv g(\mu_0)$, $\tilde\rho_0 \equiv \mu_0$. Then $\Phi(0, \mu_0, \tilde\phi_0, \tilde\rho_0) = 0$. Since the nonlinear terms in $\Phi$ are multiplied by $\tau$, the linearization at $\tau=0$ is
\[
L_0(\tilde\psi, \tilde\eta) = \big( \tilde\psi', \, \tilde\eta', \, \tilde\psi(1) - Dg(\mu_0)[\tilde\eta(1)], \, \tilde\eta(0) \big).
\]
The inverse $L_0^{-1}(f_1, f_2, \xi, \nu) = (\tilde\psi, \tilde\eta)$ is given explicitly by $\tilde\eta(\theta) := \nu + \int_0^\theta f_2$ and $\tilde\psi(\theta) := q + \int_0^\theta f_1$ where $q := \xi + Dg(\mu_0)[\nu + \int_0^1 f_2] - \int_0^1 f_1$. Arguing as before with the implicit function theorem, we deduce that $\cS$ is of class $C^1$ on $[0,T] \times \cP_0(\G)$.
\end{proof}
 When $t\leq s$, it is convenient to use the notation 
\begin{equation}\label{eq:defn-rho-v}
\rho(s, t, \mu):=\rho^{t, \mu}(s), \quad \phi(s, t, \mu):=\phi^{t, \mu}(s).
\end{equation}
%
%
\begin{lemma}\label{lem:jan24.2026.6} For $\mu \in \cP_0(\G)$ and $0\leq t \leq t_1 \leq s < T$,
\[
\rho(s, t, \mu)=\rho\big(s, t_1, \rho(t_1, t, \mu)\big), \quad  \phi(s, t, \mu)=\phi\big(s, t_1, \rho(t_1, t, \mu)\big).
\]
\end{lemma}
\proof{} Both $\big(\rho(\cdot, t, \mu), \phi(\cdot, t, \mu)\big)$ and $\big(\rho(\cdot, t_1, \rho(t_1, t, \mu)), \phi(\cdot, t_1, \rho(t_1, t, \mu))\big)$ satisfy \eqref{eq:mfg} on $[t_1, T]$ with initial measure $\rho(t_1, t, \mu)$ at time $t_1$. The result follows from uniqueness in Proposition~\ref{prop:mfg-classical}. \endproof

\begin{proof}[Proof of Theorem~\ref{thm:main-mfg}]
The existence and uniqueness of a classical solution to~\eqref{eq:mfg} is Proposition~\ref{prop:mfg-classical}. It remains to show that the value function $u(t,\mu):=\phi(t,t,\mu)$ belongs to $C^1([0,T]\times\cP_0(\G);\R^n)$ and is the unique classical solution of~\eqref{eq:master}.
By Proposition~\ref{prop:C1-dependence}, $\cS$ is $C^1$ on $[0,T]\times\cP_0(\G)$. Since $u(t,\mu)=\cS(t,\mu)(0)$ and evaluation at $\theta=0$ is a bounded linear functional on $C^1([0,1];\R^n)$, $u\in C^1([0,T]\times\cP_0(\G);\R^n)$ and $u(T,\cdot)=g$.

Fix $(t, \mu) \in (0,T) \times \cP_0(\G).$ If $h>0$ is small enough, then by Lemma \ref{lem:jan24.2026.6} 
\[
u\big(t+h, \rho(t+h, t, \mu)\big)=\phi\big(t+h, t+h, \rho(t+h, t, \mu)\big)= \phi(t+h, t, \mu).
\]
This allows us to integrate the first identity in \eqref{eq:mfg} over $[t, t+h]$ to obtain
\[
{u\big(t+h, \rho(t+h, t, \mu)\big)-u(t, \mu) \over h}={1\over h}\int_t^{t+h} \Big( H\big(\rho(\tau, t, \mu),\nabla_{\G}\phi(\tau, t, \mu)\big)-\Delta_{\G}\phi(\tau, t, \mu)\Big) d\tau.
\]
By continuity of the integrand, we obtain 
\begin{align} 
\lim_{h \to 0^+} {u\big(t+h, \rho(t+h, t, \mu)\big)-u(t, \mu) \over h} 
=& H(\mu,\nabla_{\G}u(t, \mu))-\Delta_{\G}u(t, \mu).  \label{eq:jan24.2026.4}
\end{align}
Since $\rho(t, t, \mu) = \mu$, the chain rule, integration by parts, and the second identity in \eqref{eq:mfg} give
\begin{align}
& u\big(t+h, \rho(t+h, t, \mu)\big)-u(t+h,\mu)\nonumber\\
=&- \int_t^{t+h} \Big( \nabla_\cW u\big(t+h, \rho(\tau, t, \mu)\big), B\big(\rho(\tau, t, \mu),\nabla_{\G}\phi(\tau, t, \mu)\big)+\nabla_{\G}\rho(\tau, t, \mu) \Big)d\tau\nonumber\\
=&- \int_t^{t+h} \Big( \nabla_\cW u\big(t, \rho(\tau, t, \mu)\big), B\big(\rho(\tau, t, \mu),\nabla_{\G}\phi(\tau, t, \mu)\big)+\nabla_{\G}\rho(\tau, t, \mu) \Big)d\tau
\label{eq:jan24.2026.7}\\
&- \int_t^{t+h} \Big(a(\tau, t, h, \mu) , B\big(\rho(\tau, t, \mu),\nabla_{\G}\phi(\tau, t, \mu)\big)+\nabla_{\G}\rho(\tau, t, \mu) \Big)d\tau,\nonumber
\end{align}
where $a(\tau, t, h, \mu):=\nabla_\cW u\big(t+h, \rho(\tau, t, \mu)\big)-\nabla_\cW u\big(t, \rho(\tau, t, \mu)\big)$. Since $u \in C^1$,
\[
\lim_{h \to 0} {1\over h} \int_t^{t+h} \Big(a(\tau, t, h, \mu) , B(\rho(\tau, t, \mu),\nabla_{\G}\phi(\tau, t, \mu))+\nabla_{\G}\rho(\tau, t, \mu) \Big)d\tau=0
\]
and so, by continuity of the integrand in \eqref{eq:jan24.2026.7}, we conclude that 
\begin{align} 
 \lim_{h \to 0}  {u\big(t+h, \rho(t+h, t, \mu)\big)-u(t+h,\mu)\over h}
=&- \Big( \nabla_\cW u(t, \mu), B(\mu ,\nabla_{\G}u(t, \mu))+\nabla_{\G}\mu\Big)\label{eq:jan24.2026.5}
\end{align}
Combining \eqref{eq:jan24.2026.4} and \eqref{eq:jan24.2026.5}, we have 
\begin{align*} 
\partial_t u(t, \mu)=&\lim_{h \to 0^+} {u(t+h,\mu)-u\big(t+h, \rho(t+h, t, \mu)\big) +u\big(t+h, \rho(t+h, t, \mu)\big)-u(t,\mu)\over h}\\
=& \big( \nabla_\cW u(t, \mu), B(\mu ,\nabla_{\G}u(t, \mu))\big)-\Delta_\text{ind}u(t,\mu)+ H(\mu,\nabla_{\G}u(t, \mu))-\Delta_{\G}u(t, \mu),
\end{align*}
where we used $\Delta_\text{ind}u(t,\mu)=-(\nabla_\cW u(t,\mu), \nabla_\G\mu)$.

Next, we prove uniqueness. Let $v \in C^1([0,T] \times \cP_0(\G); \R^n)$ be another classical solution to \eqref{eq:master}. Fix $(t, \mu) \in [0,T) \times \cP_0(\G)$ and let $(\phi(\cdot), \rho(\cdot)):=(\phi(\cdot,t,\mu),\rho(\cdot,t,\mu))$ be the unique classical solution of \eqref{eq:mfg} on $[t,T]$ with initial measure $\mu$ at time $t$.
Moreover, we define the function $z:[t,T] \to \R^n$ by
\[z(s):=v(s, \rho(s)).\]
Using the chain rule, we compute for $s \in (t,T)$,
\begin{align*}
\dot z(s)= &\partial_t v(s, \rho(s)) +  \delta_\mu v (s, \rho(s))[\dot \rho(s)]\\
= &\partial_t v(s, \rho(s)) - \big( \nabla_\cW v(s, \rho(s)), B(\rho(s),\nabla_{\G}\phi(s))+\nabla_{\G}\rho(s) \big).
\end{align*}
Now, using the fact that $v$ solves the master equation, we obtain
\begin{align*}
\dot z(s) =& \big(\nabla_\cW v, B(\rho, \nabla_\G v)\big) - \Delta_{\text{ind}} v + H(\rho, \nabla_\G v) - \Delta_\G v - \big(\nabla_\cW v, B(\rho, \nabla_\G\phi) + \nabla_\G\rho\big)\\
= & \big(\nabla_\cW v, B(\rho, \nabla_\G v)\big) + \big(\nabla_\cW v, \nabla_\G\rho\big) + H(\rho, \nabla_\G v) - \Delta_\G v - \big(\nabla_\cW v, B(\rho, \nabla_\G\phi) + \nabla_\G\rho\big)\\
= & \big(\nabla_\cW v, B(\rho, \nabla_\G v) - B(\rho, \nabla_\G\phi)\big) + H(\rho, \nabla_\G v) - \Delta_\G v.
\end{align*}
Letting $\eta(s)=z(s)-\phi(s)$ for $s \in [t,T]$ and $\Gamma(s)=\nabla_{\cW}v(s, \rho(s))$, we have $\eta(T)=0$ and
\begin{align*}
\dot \eta(s) &= \big(\Gamma, B(\rho, \nabla_\G\phi + \nabla_\G\eta) - B(\rho, \nabla_\G\phi)\big) + H(\rho, \nabla_\G\phi + \nabla_\G\eta) - H(\rho, \nabla_\G\phi) - \Delta_\G\eta=: f(s, \eta(s)).
\end{align*}
Since $\eta \equiv 0$ is a solution to this ODE (with terminal condition $\eta(T) = 0$), it must be the only solution due to the uniqueness coming from the regularity of $f$. Hence $v(s, \rho(s)) = \phi(s)$ for all $s \in [t,T]$. In particular, $v(t, \mu) = \phi(t) = u(t, \mu)$.
\end{proof}
%
%
%
%
\section{Classical solutions to the HJB equation on \texorpdfstring{$\cP_0(\G)$}{P\_0(G)}}\label{sec:ClassicalHJE}
%
\subsection{Assumptions}\label{subsection:assumptions-g}
Throughout this section, $\cL:[0,+\infty)^n\times\bS(n)\to\bR\cup\{+\infty\}$, $\cF:[0,1]^n\to\bR$, and $\cU_T:[0,1]^n\to\bR$ are assumed to satisfy the hypotheses collected below, in which $C_\cL>0$ is a constant, $p_0>1$ is a growth exponent with conjugate $p_0':=p_0/(p_0-1)$, and $a:(0,\infty)\times\bS(n)\to[0,\infty)$ is a locally bounded function. On the Lagrangian $\cL$:
\begin{align}
& \cL \in C^\infty\big((0,+\infty)^n \times \bS(n)\big), \quad \cL(\mu,  0)=0, \label{eq:assump2L}\\
& \cL\; \text{is (jointly) convex and (jointly) lower semicontinuous},\label{eq:assump1L}\\
& D_{mm}^2 \cL(\mu, m)>0\quad \forall (\mu, m) \in (0,1)^n \times \bS(n),\label{eq:assump1LDec2025}\\
& \cL(\mu, m)\geq C_{\cL}\big(\|m\|_{\ell_2}^{p_0}-1\big)\quad \forall (\mu,m)\in \cP(\bG)\times \bS(n),\label{eq:may24.2025.3}\\
& (m, D_m \cL(\mu, m)) + (\mu, D_\mu \cL(\mu, m)) \;\text{and}\; \min_i -D_{\mu^i} \cL(\mu, m)\; \text{are bounded from below}.\label{eq:assump2Llower}
\end{align}
On the Hamiltonian $\cH(\mu,\cdot)$, defined as the Legendre transform of $\cL(\mu,\cdot)$:
\begin{align}
& \cH \in C^\infty((0,1)^n \times \bS(n)),\label{eq:dec20.2025.2}\\
& |D_{p^{ij}} \cH(\mu, p)|\leq (\mu^i+\mu^j) a\Big({\mu^j \over \mu^i},p \Big)\quad \forall (\mu, p) \in (0,1)^n \times \bS(n),\label{eq:dec21.2025.1}\\
& \lim_{u\to 0}a(u,p) = \lim_{u\to \infty}a(u,p)=0 \;\text{ locally uniformly in } p\in \bS(n).\label{eq:dec21.2025.3}
\end{align}
On the source $\cF$ and terminal cost $\cU_T$:
\begin{align}
& \cF \in C^2([0,1]^n) \cap C^\infty((0,1)^n), \quad D^2 \cF < 0 \text{ on } (0,1)^n,\label{eq:assumpF}\\
& \cU_T\in C^2([0,1]^n) \cap C^\infty((0,1)^n), \quad \cU_T|_{\cP(\bG)} \text{ is convex}.\label{eq:assumpcUT}
\end{align}

For instance, $\cL_\theta$ satisfies \eqref{eq:assump2L}, \eqref{eq:assump1L} and the first identity in \eqref{eq:assump2Llower}. The fact that $\partial_1 \theta \geq 0$ implies that $\cL_\theta$ satisfies the second identity in \eqref{eq:assump2Llower}. Finally, \eqref{eq:dec21.2025.2} yields \eqref{eq:dec21.2025.1}.

By Legendre duality, for every $\mu\in\cP_0(\G)$ and $w\in\bS(n)$, setting $p:=D_m\cL(\mu,w)$ yields
\begin{equation}\label{eq:aug23.2025.1}
D_\mu\cL(\mu,w)=-D_\mu\cH(\mu,p).
\end{equation}
\begin{remark} \label{rem:dec20.2025.1}
For every $\varrho \in (0,1)^n$, $p \in \bS(n)$, and $(\eta, q) \in (\bR^n_0 \times \bS(n)) \setminus \{(0,0)\}$,
\begin{equation}\label{eq:strong monotonicity from assumptions}
(\eta, D_{\mu\mu} \cH(\varrho,p)[\eta] + D^2\cF(\varrho)[\eta]) < (q, D_{pp}\cH(\varrho,p)[q]).
\end{equation}
Indeed, \eqref{eq:assump1LDec2025} and Legendre duality give $D_{pp}\cH=(D_{mm}^2\cL)^{-1}>0$; joint convexity~\eqref{eq:assump1L} and preservation of convexity under partial minimization give $D_{\mu\mu}\cH\leq 0$; and \eqref{eq:assumpF} gives $D^2\cF<0$. For $\eta\neq 0$, the left-hand side is at most $(\eta,D^2\cF[\eta])<0\leq(q,D_{pp}\cH[q])$; for $\eta=0$, the left-hand side vanishes while $q\neq 0$ makes the right-hand side positive.
\end{remark}

We shall apply the theory developed in Section \ref{sec:classical} by setting $F:= D_\mu \cF$, $U_T:= D_\mu \cU_T$, and
\begin{equation}\label{eq:define-all-functions}
H(\mu, p):=D_\mu\cH(\mu, -p)+F(\mu), \;\; B(\mu, p)=-D_p\cH(\mu, -p), \;\; g(\mu)=U_T(\mu).
\end{equation}
By \eqref{eq:assumpF} and \eqref{eq:assumpcUT}, $F$ and $U_T$ are of class $C^1$, so by \eqref{eq:dec20.2025.2} we have $H \in C^1((0,1)^n \times \bS(n);\R^n)$ and $B\in C^1((0,1)^n\times \bS(n);\bS(n))$. Furthermore, $g = U_T \in C^1([0,1]^n)$. By Remark~\ref{rem:dec20.2025.1}, condition~\eqref{eq:strong monotonicity from assumptions} holds; since $(\eta, DF[\eta]) = (\eta, D^2\cF[\eta])$ for $\eta \in \bR^n_0$, this implies \eqref{as:uniqueness 3} and hence \eqref{as:uniqueness 1}. Assumption \eqref{eq:assump2Llower} implies that \eqref{as:B coercivity} and \eqref{as:H lower bound} hold, and \eqref{eq:mobility-dominated flux} follows from \eqref{eq:dec21.2025.1}.
\begin{remark}[H\"older regularity from coercivity]\label{rem:dec22.2025.1}
Assume \eqref{eq:may24.2025.3}. Let $(\rho,m)\in \cC_t^T(\mu,\cdot)$ and set $w:=m+\nabla_\bG\rho$. Then
\[
\int_t^T \|w(s)\|_{\ell_2}^{p_0}\,ds
\le \frac{1}{C_{\cL}}\int_t^T \cL(\rho(s),w(s))\,ds + (T-t).
\]
Since $\rho(s)\in \cP(\bG)$ for all $s$, $\|\nabla_\bG\rho(s)\|_{\ell_2}$ is bounded uniformly in $s$ by a constant depending only on $(\bG,\omega)$; hence
\[
\int_t^T \|m(s)\|_{\ell_2}^{p_0}\,ds
\le C\Big(1+\int_t^T \cL(\rho(s),m(s)+\nabla_\bG\rho(s))\,ds\Big),
\]
for a constant $C$ depending only on $p_0,C_{\cL},T$ and the graph data.
Moreover, the continuity equation implies $\dot\rho=-\nabla_\bG\cdot m$ in the distributional sense and thus
$\|\dot\rho(s)\|_{\ell_2}\le C_{\rm div}\|m(s)\|_{\ell_2}$ for a.e.\ $s$, where $C_{\rm div}$ depends only on $(\G,\omega)$.
Consequently $\dot\rho\in L^{p_0}(t,T;\ell_2)$ and for all $s_1,s_2\in[t,T]$,
\[
\|\rho(s_2)-\rho(s_1)\|_{\ell_2}
\le |s_2-s_1|^{1/p_0'}\|\dot\rho\|_{L^{p_0}(t,T;\ell_2)}.
\]
In particular, $\rho$ is $1/p_0'$--H\"older continuous in time.
\end{remark}

%
%
\subsection{Existence and uniqueness of a minimizer for \texorpdfstring{$\cU(t, \mu)$}{U(t,mu)}}
We set  
\begin{equation}\label{eq:june15.2025.1}
\iota_0:=-\max \cU_T^- -T \max |\cF| - TC_{\cL}, \qquad \iota^T:=T \max_{\cP(\bG)} |\cF|+ \max_{\cP(\bG)} \cU_T.
\end{equation}

\begin{proposition}[Existence of a unique minimizer in $\cC_t^T(\mu, \cdot)$]\label{prop:action-minimizer} Let $\mu \in \cP(\bG)$ and  $t\in [0,T).$ Then, the following hold: 
\begin{enumerate}
\item[(i)]  $ \iota_0 \leq \cU(t, \mu) \leq \iota^T$.
\item[(ii)] There exists a unique $(\rho, m)$ that minimizes $(\bar \rho, \bar m) \to \cA_t^T(\rho, \bar m)+\cU_T(\bar \rho(T))$ over $\cC_t^T(\mu, \cdot)$.  
\item[(iii)] $\cU(t, \cdot)$ is strictly convex on $\cP(\bG).$
\item[(iv)] There is a constant $R>0$ depending only on  $\max |\cF|$, $\max |\cU_T|$, $p_0$, $C_{\cL}$ and $T$ (but independent of $t$ or $\mu$) such that if   $(\rho, m)$ is as in (ii), then $\|\dot \rho\|_{L^{p_0}(t, T)} \leq R.$
\end{enumerate}
\end{proposition}
\proof{} Let $\mu \in \cP(\bG)$ and  $t\in [0,T).$ Using \eqref{eq:may24.2025.3}, we obtain the first identity in (i). We set 
\[
A^{ij}= \left\{
     \begin{array}{lr}
       \quad  \omega_{ij}& \text{if} \quad \qquad j\in N(i)  \\
       \quad 0 & \quad \text{if} \;  j\not \in N(i), i \not =j\\
       -\sum_{k \in N(i)} \omega_{ik}  & \text{if} \; \qquad \qquad i=j.
     \end{array}
   \right.
\]
Setting 
\[
\rho(s)=\mu e^{(s-t)A}, \qquad m:=- \nabla_\bG \rho \qquad \forall s \in [t, T], 
\]
we find that $\rho(t)=\mu$ and 
\[\dot \rho-\nabla_\bG \cdot \big( \nabla_\bG \rho \big)=0, \quad \text{and} \quad m+ \nabla_\bG \rho=0.\]
By the second identity in \eqref{eq:assump2L}, $\cL(\rho,  m+ \nabla_\bG \rho)=0$ and thus  
\begin{equation}\label{eq:aug1.2025.4new}
\cU(t, \mu) \leq \cA_t^T(\rho,  m)+\cU_T(\rho_T)= -\int_t^T \cF(\rho)ds+\cU_T(\rho_T) \leq \iota^T,
\end{equation}
which proves the second identity in (i).

The convexity of $\cL$ and $\cU_T$, together with the strict convexity of $-\cF$, ensures that $(\bar \rho, \bar m) \to \cA_t^T(\bar \rho, \bar m)+\cU_T(\bar \rho(T))$ is strictly convex over $\cC_t^T(\mu, \cdot)$, and thus there is at most one minimizer. Combined with the existence of a minimizer (established below), this implies (iii).

Let $(\rho_k, m_k)_k \subset \cC_t^T(\mu, \cdot)$ be a minimizing sequence and assume without loss of generality that $\cA_t^T(\rho_k, m_k)+\cU_T(\rho_k(T)) \leq \cU(t, \mu)+1.$ From the boundedness of $\cF$ and $\cU_T^-$, we deduce an upper bound on $\int_t^T \cL(\rho_k, m_k+\nabla_\bG\rho_k)\,ds$. By the coercivity assumption \eqref{eq:may24.2025.3} and Remark~\ref{rem:dec22.2025.1}, $(m_k)_k$ is bounded in $L^{p_0}\big((t, T);\bS(n)\big)$ and $(\dot\rho_k)_k$ is bounded in $L^{p_0}\big((t, T);\ell_2\big)$. Consequently, the $1/p_0'$-H\"older constant of each $\rho_k$ is bounded by a constant independent of $k$. Passing to a subsequence if necessary, we apply the Arzel\`a--Ascoli theorem to obtain that  $(\rho_k)_k$ is uniformly convergent to some $\rho.$ Furthermore, since $p_0>1$, we may assume that $(m_k)_k$ converges weakly to some $m$ in $L^{p_0}\big((t, T);\bS(n)\big)$. We have that $(\rho, m) \in \cC_t^T(\mu, \cdot)$. We use the joint convexity of $\cL$ and the fact that $\cL$ is bounded from below  to conclude that
\[
\liminf_{k \to+\infty} \int_t^T\cL(\rho_k, m_k+  \nabla_\bG \rho_k)ds \geq  \int_{t}^T \Big(\cL(\rho, m+  \nabla_\bG \rho)\Big)ds.
\]
Since  $\cF$ and $\cU_T$ are continuous, we conclude that $(\rho, m)$ is a minimizer, which proves (ii). As written above, this also proves (iii). By Remark~\ref{rem:dec22.2025.1}, we find an upper bound for $\|\dot \rho\|_{L^{p_0}(t, T)}$ which depends only on  $\max |\cU|$, $\max |\cF|$, $\max |\cU_T|$, $p_0$, $C_{\cL}$ and $T$. This, together with (i), implies (iv).
\endproof
\begin{proposition}[Characterization of minimizing paths in $\cC_t^T(\mu, \cdot)$]\label{prop:sufficient-min} Given $\epsilon>0$ and $\mu \in \cP_\epsilon(\G)$, there exists a unique $(\rho, m)$ minimizing $(\bar \rho, \bar m) \to \cA_t^T(\bar \rho, \bar m)+\cU_T(\bar \rho(T))$ over $\cC_t^T(\mu, \cdot)$. Furthermore:
\begin{enumerate}
\item[(i)] $(\rho, m)$ is uniquely characterized as follows: $\rho\in C^\infty([t,T];(0,1)^n)$, and there exists $\phi\in C^\infty([t,T];\R^n)$ such that $m:= D_p \cH(\rho, -\nabla_{\G}\phi)-\nabla_\bG \rho$ and the pair $(\phi,\rho)$ solves the forward--backward system~\eqref{eq:mfg-withcH}.
\item[(ii)] There exists $c_\epsilon>0$ independent of $\mu$ and $t$, but depending on $\epsilon$ and $T$, such that the range of $\rho$ is contained in $\cP_{c_\epsilon}(\G).$
\item[(iii)] There exists a constant $C$ independent of $\mu$, $\epsilon$, and $t$, but depending on the data, such that $\|\phi\|_{\ell_\infty} \leq C/\epsilon$ on $[t, T].$
\end{enumerate}
\end{proposition}
\proof{} By Proposition~\ref{prop:mfg-classical}, the system \eqref{eq:mfg-withcH} admits a unique solution $(\phi, \rho)$. Set $m:= D_p \cH(\rho, -\nabla_{\G}\phi)-\nabla_\bG \rho$. By standard convex analysis (see \eqref{eq:aug23.2025.1}), 
\[
\cL(\rho, m+\nabla_\bG \rho)+\cH(\rho, -\nabla_{\G}\phi)=(m+\nabla_{\G}\rho, -\nabla_{\G}\phi),
\]
and 
\begin{equation}  \label{eq:dec17.2025.1}
 D_m \cL(\rho, m+\nabla_\bG \rho)=-\nabla_{\G}\phi \quad \text{and} \quad D_\mu  \cL(\rho, m+\nabla_\bG \rho)= - D_\mu \cH(\rho, -\nabla_{\G}\phi)
\end{equation}
Let $(\bar \rho, \bar m) \in \cC_t^T(\mu, \cdot).$ Using the fact that $\cL$ and $-\cF$ are convex, we infer 
 \begin{align*}
 \cA_t^T(\bar \rho,  \bar m)  \geq   \cA_t^T( \rho, m)
 +&\int_t^T \big( D_m \cL(\rho, m+\nabla_\bG \rho), \bar m+\nabla_\bG \bar \rho-m-\nabla_\bG \rho\big)ds\\
  +& \int_t^T \Big( \big( D_\mu  \cL(\rho, m+\nabla_\bG \rho), \bar \rho -\rho\big) - \big(\delta_\mu \cF(\rho), \bar \rho-\rho\big)\Big)ds.
 \end{align*}
Since $F-\delta_\mu \cF$ lies in the span of ${\bf 1}$ and $\bar \rho - \rho \in \bR^n_0$, we may replace $\delta_\mu \cF$ by $F$. Combining this with the second identity in \eqref{eq:mfg-withcH} and \eqref{eq:dec17.2025.1}, we conclude that  
 \begin{align*}
 \cA_t^T(\bar \rho,  \bar m)  \geq   \cA_t^T(\rho,  m)
 -&\int_t^T \big(\nabla_{\G}\phi, \bar m-m+\nabla_\bG \bar \rho-\nabla_\bG \rho\big) ds\\
  -& \int_t^T\Big(\big(\dot \phi +\Delta_\G \phi-F(\rho) , \bar \rho -\rho\big)+ \big(F(\rho), \bar \rho-\rho\big) \Big) ds\\
 = \cA_t^T(\rho, m) -&  \int_t^T\Big( \big(\nabla_{\G}\phi, \bar m-m\big) +
 \big(\dot \phi, \bar \rho -\rho\big)\Big) ds
 \end{align*}
Using the fact that $(\bar \rho, \bar m)$ and $(\rho, m)$ belong to $\cC_t^T(\mu, \cdot)$, we obtain 
\[
\cA_t^T(\bar \rho,  \bar m)  \geq  \cA_t^T(\rho, m) - \big(\phi(T), \bar \rho(T)-\rho(T)\big) = \cA_t^T(\rho, m)- \big(U_T(\rho(T)), \bar \rho(T)-\rho(T) \big).
\]
%
Since $\cU_T$ is convex and 
\[
\big(\delta_\mu \cU_T(\rho(T)), \bar \rho(T)-\rho(T) \big)=  \big(U_T(\rho(T)), \bar \rho(T)-\rho(T) \big),
\]
we conclude that 
\[
 \cA_t^T(\bar \rho,  \bar m)+\cU_T(\bar \rho(T))  \geq \cA_t^T( \rho,  m)+\cU_T( \rho(T)).
\]
This proves that $(\rho, m)$ minimizes $(\bar \rho, \bar m) \mapsto \cA_t^T(\bar \rho, \bar m)+\cU_T(\bar \rho(T))$ over $\cC_t^T(\mu, \cdot)$. By Proposition \ref{prop:action-minimizer} (ii), $(\rho, m)$ is uniquely determined, proving (i). For (ii) and (iii), Lemma~\ref{lem:rho lower bound mfg lambda} and Proposition~\ref{prop:phi bd} applied to~\eqref{eq:mfg-withcH} on $[t,T]$ yield constants that depend monotonically on $T$, hence uniform in $t\in[0,T)$. \endproof

%
%

%
%
\subsection{Properties of the subdifferential of \texorpdfstring{$\cU(t, \cdot)$}{U(t,.)}}

We denote by $\partial_\cdot \cU(t, \cdot)(\mu)$ the subdifferential of $\cU(t, \cdot)$ with respect to the $\ell_2$--metric. Since $\cU(t, \cdot)$ is convex, this is the set of $q \in \bR^n_0$ such that  
\[
\cU(t, \nu)-\cU(t, \mu) \geq (\nu-\mu, q), \qquad \forall \nu \in \cP(\bG).
\]
When $\partial_\cdot \cU(t, \cdot)(\mu)$ contains a unique element, we denote it by $\delta_{\mu} \cU (t, \mu).$ 
\begin{lemma}\label{lem:bound-on-m-rho-dotb} There exists a constant $\gamma$ (which depends only on the data, including $p_0$ and $C_{\cL}$) such that the following holds:  if $\mu \in \cP(\bG)$,  $t\in [0,T)$, $(\rho, m) \in \cC_t^T(\mu, \cdot)$, and $\cA_t^T(\rho, m)+ \cU_T(\rho(T)) \leq \cU(t, \mu)+1$, then the following hold:
\[
\int_t^T \|m\|^{p_0}_{\ell_2} ds \leq \gamma\big(1 +|\cU(t, \mu)|\big), \qquad \int_t^T|\dot \rho|^{p_0} ds \leq C_{\rm div}^{p_0} \gamma\big(1 +|\cU(t, \mu)|\big),
\]
where $C_{\rm div}$ depends only on $(\G,\omega)$.
Therefore, the $1/p_0'$-H\"older constant of $\rho$ is bounded by a constant $R_{\rm data}$ independent of $t$ and $\mu$ that depends only on the data.
\end{lemma}
\proof{} From the hypothesis $\cA_t^T(\rho, m)+ \cU_T(\rho(T)) \leq \cU(t, \mu)+1$ and the boundedness of $\cF$ and $\cU_T^-$, we deduce an upper bound on $\int_t^T \cL(\rho, m+\nabla_\bG\rho)\,ds$ in terms of $1+|\cU(t,\mu)|$. By the coercivity assumption \eqref{eq:may24.2025.3} and the argument in Remark~\ref{rem:dec22.2025.1}, we obtain the first inequality. For the second, we use that by \eqref{eq:may-continuity.ter}, $\|\dot \rho\|_{\ell_2} \leq C_{\rm div} \|m\|_{\ell_2}$ where $C_{\rm div}$ depends only on $(\G,\omega)$. Since Proposition \ref{prop:action-minimizer} provides an upper bound on $\cU$, the $1/p_0'$-H\"older constant $R_{\rm data}$ is bounded independently of $t$ and $\mu$. \endproof

\begin{proposition}\label{prop:action-frechet} Let $\epsilon>0$ and $\mu \in \cP_\epsilon(\G)$. Then
\begin{enumerate}
\item[(i)]  $\partial_\cdot   \cU(t, \cdot)(\mu)\not=\emptyset$ and $\cU(t, \cdot)$ is Fr\'echet differentiable almost everywhere on $\cP_0(\G)$.
\item[(ii)] If $q \in \partial_\cdot \cU(t, \cdot)(\mu)$ and $\mu \in \cP_\epsilon(\G)$, then $ \epsilon\|q\|_{\ell_2} \leq 2\big(\iota^T-\iota_0\big)$ for all $t\in [0,T].$
\end{enumerate}
\end{proposition}
\proof{} (i) By Proposition \ref{prop:action-minimizer}, $\cU(t, \cdot)$ is convex and thus, since it assumes only finite values, locally Lipschitz continuous. Since $\cP_0(\G)$ is an open subset of the affine hyperplane $\{\nu\in\bR^n:\sum_i\nu_i=1\}$, part (i) follows from Rademacher's theorem. (ii) Suppose that $q \in \partial_\cdot \cU(t, \cdot)(\mu)$ and assume without loss of generality that $q \not =0.$ Then 
\[
\cU\bigg(t, \mu+{\epsilon \over 2 \|q\|_{\ell_2}}q \bigg)-\cU(t, \mu) \geq \bigg({\epsilon \over 2 \|q\|_{\ell_2}}q, q\bigg)={\epsilon \over 2}\|q\|_{\ell_2}.
\]
We use Proposition \ref{prop:action-minimizer} (i) to conclude the proof of (ii).  \endproof

\begin{proposition}\label{thm:main2} Let $\epsilon>0$, $\mu \in \cP_\epsilon(\G)$, and let $(\rho, m)$ be the unique minimizer from Proposition \ref{prop:sufficient-min}. Let $\phi$ be as in that proposition, so that $m+\nabla_\bG \rho= D_p \cH(\rho, -\nabla_{\bG}\phi)$, and set $p:=-\nabla_\bG \phi$. Then:
\begin{enumerate}
\item[(i)] For any $\tau \in [t, T)$, $\cU(\tau, \cdot)$ is differentiable at $\rho({\tau})$.
\item[(ii)] $p$ satisfies the initial value problem
 \begin{equation}\label{eq:june18.2025.3}
\dot p= -\nabla_{\cW} \cH(\rho, p)-\nabla_{\cW} \cF(\rho)- \nabla_\bG (\nabla_\bG \cdot p) \quad \text{on} \; (t, T), \qquad \nabla_{\cW}\cU(t, \mu)=-p(t).
 \end{equation}
 Hence, $\nabla_{\cW}\cU(\cdot, \rho(\cdot))=\nabla_{\bG}\phi$ holds on $[t,T]$, and thus these functions are continuous.
\end{enumerate}
\end{proposition}
 \proof{}
Set $\tilde \cL(\mu,w):=\cL(\mu,w)-\cF(\mu)$, so that
\[
\cA_{\tau}^{T}(\rho,m)=\int_\tau^{T}\tilde \cL\big(\rho,\,m+\nabla_\bG\rho\big)\,ds .
\]

Fix $\tau\in[t,T)$. By Proposition~\ref{prop:action-frechet}(i), there exists $q\in \partial_\cdot \cU(\tau,\cdot)(\rho(\tau))$. Take $A\in C^\infty([\tau,T];\bS(n))$ with $A(T)=0$, and set
\[
f:=-\nabla_\bG\cdot A,\qquad \rho^r:=\rho+r f,\qquad m^r:=m+r\dot A \quad \text{on }[\tau,T].
\]
For $|r|$ small, $(\rho^r,m^r)\in\cC_\tau^{T}(\rho^r(\tau),\cdot)$ with $\rho^r(T)=\rho(T)$. By optimality of $(\rho,m)$ and the subgradient inequality for $q$,
\[
(q,\rho^r(\tau)-\rho(\tau))
\le \cA_\tau^{T}(\rho^r,m^r)-\cA_\tau^{T}(\rho,m).
\]
Differentiating in $r$ at $r=0$ and using $f=-\nabla_\bG\cdot A$, $(u,-\nabla_\bG\cdot A)=(\nabla_\bG u,A)$, we obtain
\begin{equation}\label{eq:EL-main2}
0=\int_\tau^{T}\Big(
(\nabla_\cW \tilde \cL(\rho,w),A)
+(D_m\tilde \cL(\rho,w),\dot A-\nabla_\bG(\nabla_\bG\cdot A))
\Big)\,ds
-(\nabla_\bG q,A(\tau)).
\end{equation}
By the Legendre relations \eqref{eq:dec17.2025.1}, integrating by parts in time (using $A(T)=0$), and applying the graph integration by parts formula, we obtain
\[
0=\int_\tau^{T}
\big(
-\nabla_\cW \cH(\rho,p)-\nabla_\cW \cF(\rho)
-\dot p-\nabla_\bG(\nabla_\bG\cdot p)
,\;A
\big)\,ds
-\big(p(\tau)+\nabla_\bG q,\;A(\tau)\big).
\]
Choosing $A$ supported in $(\tau,T)$ gives
\[
\dot p
= -\nabla_\cW \cH(\rho,p)-\nabla_\cW \cF(\rho)-\nabla_\bG(\nabla_\bG\cdot p)
\qquad\text{on }(\tau,T),
\]
and allowing arbitrary $A(\tau)$ gives the boundary identity
\begin{equation}\label{eq:main2-boundary}
p(\tau)=-\nabla_\bG q.
\end{equation}
Since $\bG$ is connected, \eqref{eq:main2-boundary} determines $q\in \mathbb{R}^n_0$ uniquely as a smooth function of $\tau$. Hence, since the subdifferential is a singleton, $\cU(\tau,\cdot)$ is differentiable at $\rho(\tau)$, and $\nabla_\cW \cU(\tau,\rho(\tau))=-p(\tau)$. Since $\tau\in[t,T)$ was arbitrary, the differential equation holds on $(t,T)$ with initial condition at $\tau=t$, proving (ii). The continuity of $\nabla_\cW\cU(\cdot,\rho(\cdot))=\nabla_\bG\phi$ follows from $\phi\in C^\infty([t,T])$. 
\endproof
\subsection{\texorpdfstring{$C^{1,1}$-properties of $\cU(t, \cdot)$ on compact subsets of $\cP_0(\G)$}{C\^{}\{1,1\}-properties of U(t,.) on compact subsets of P\_0(G)}}\label{subsec:LocalSemi}

In this subsection, we fix $\epsilon>0$, $\mu \in \cP_\epsilon(\G)$, $t\in [0,T)$, and let $(\rho, m)$ be the minimizer from Proposition \ref{prop:action-minimizer}.
\begin{remark}\label{rem:bound-on-m} By Proposition \ref{prop:action-frechet} and \eqref{eq:june18.2025.3} in Proposition \ref{thm:main2}, there exists a constant $C(T)$ (depending on $T$ and the data but independent of $t$, $\epsilon$, and $\mu$) such that
\[
\epsilon \|m+ \nabla_\bG \mu\|_{\ell_\infty}=\epsilon\|D_p\cH\big(\mu, -\nabla_{\cW}\cU(t, \mu)\big)\|_{\ell_\infty}\leq C(T).
\]
\end{remark}
\begin{proposition}[Spatial semiconcavity]\label{thm:semi-concave} There exists a constant $C_\epsilon$ (depending on $\epsilon$ and the data, but independent of $\mu$ and $t$) such that
\[
\cU(t, \mu+h)+ \cU(t, \mu-h)-2 \cU(t, \mu) \leq C_\epsilon \|h\|^2_{\ell_2},
\]
whenever $h \in \bR^n_0$ is such that $\|h\|_{\ell_2}\ll1$.
\end{proposition}
\proof{}  We assume without loss of generality that $\epsilon \in (0,1).$ By Proposition \ref{prop:sufficient-min} (ii), there exists $c_\epsilon>0$ (depending on $\epsilon$ and $T$, but independent of $\mu$ and $t$) such that $\rho([t,T]) \subset \cP_{c_\epsilon}(\G).$ If
\begin{equation}\label{eq:dec26.2025.2}
\|h\|_{\ell_\infty} \leq  \min\Big\{ {c_\epsilon\over 4}, {C(T)\over c_{\epsilon}\sqrt{\omega_{\max}}}\Big\},
\end{equation}
then $\rho([t, T])\pm h \subset \cP_{c_\epsilon/2}(\bG)$. For each $\tau \in [t, T)$, the restriction $(\rho, m)|_{[\tau, T]}$ is optimal for $\cU(\tau, \rho(\tau))$. Hence, applying Remark \ref{rem:bound-on-m} with $\epsilon = c_\epsilon$ to the minimizer starting from $(\tau, \rho(\tau))$ yields $(m + \nabla_\bG \rho)([t, T]) \subset \bB_{2C(T)/c_\epsilon}$. Since $\|\nabla_\bG h\|_{\ell_\infty} \leq 2\sqrt{\omega_{\max}}\|h\|_{\ell_\infty}$, we have $(m+ \nabla_\bG \rho\pm \nabla_\G h)([t, T]) \subset \bB_{4C(T)/c_\epsilon}.$ Set $\rho^{\pm h}:=\rho\pm h$. Let $\tilde \cL:=\cL-\cF$ and let $\bar c_\epsilon$ bound the largest eigenvalues of $D^2\tilde \cL$ over $\cP_{c_\epsilon/2}(\G) \times \bB_{4C(T)/c_\epsilon}$ and of $D^2\cU_T$ over $\cP_{c_\epsilon/2}(\G).$ Since $(\rho^{\pm h}, m) \in \cC_t^T(\mu\pm h, \cdot)$, we have $\cU(t, \mu\pm h)\leq \cA_t^{T}(\rho^{\pm h}, m) +\cU_T(\rho^{\pm h}(T)).$ Using optimality of $(\rho, m)$ and the bound on $D^2\tilde \cL$ and $D^2\cU_T$,
\begin{multline*}
\cU(t, \mu+ h)+\cU(t, \mu- h) -2 \cU(t, \mu) \leq  \cA_t^{T}(\rho^{+h}, m)+\cA_t^{T}(\rho^{-h}, m) -2\cA_t^{T}(\rho, m)\\
 + \cU_T(\rho^{+h}(T))+\cU_T(\rho^{-h}(T)) -2\cU_T(\rho(T)) \leq  \bar c_{\epsilon}\bigg( \int_t^T \big(\|h\|^2_{\ell_2}+ \|\nabla_\bG h\|^2_{\ell_2}\big)ds +\|h\|^2_{\ell_2}\bigg).
\end{multline*}
\endproof
\begin{corollary}\label{cor:continuous} Given $\epsilon>0$, there exists $C_\epsilon>0$ such that the first and second derivatives of $\cU(t, \cdot)$ are bounded on $\cP_\epsilon(\G)$ by $C_\epsilon$ for all $t \in [0, T].$ 
\end{corollary}
\proof{} Proposition \ref{prop:action-frechet} provides an upper bound, independent of $t$, on $\|\delta_{\mu} \cU(t, \cdot)\|_{\ell_2}$ restricted to $\cP_\epsilon(\G)$. By Proposition \ref{prop:action-minimizer}, $\cU(t, \cdot)$ is convex on $\cP_0(\G)$ for all $t \in [0, T]$, and thus the eigenvalues of the $\ell_2$ second derivatives of $\cU(t, \cdot)$ are nonnegative. Since by Proposition \ref{thm:semi-concave} these eigenvalues are bounded from above on $\cP_\epsilon(\G)$ by a constant independent of $t$, the result follows. \endproof

%
%
%
%
%
\subsection{\texorpdfstring{$C^2$-regularity}{C2-regularity} of the value function in both variables}\label{sec:visc}
\begin{remark} \label{rem:diff_char_val} Fix $t\in[0,T)$ and $\mu \in \cP_0(\G)$, and let $(\rho,m)$ be the optimizer for $\cU(t,\mu)$. Since the restriction of $(\rho,m)$ to $[\tau,T]$ is optimal for $\cU(\tau,\rho(\tau))$, we have
\[
\cU(\tau,\rho(\tau)) = \int_\tau^T \big[\cL(\rho, m + \nabla_\bG\rho) - \cF(\rho)\big]\,ds + \cU_T(\rho(T)).
\]
By Proposition \ref{prop:sufficient-min}, $(\rho,m)$ is smooth, so differentiating gives
\[
-\left.\frac{d}{d\tau}\right|_{\tau = t^+}\cU(\tau,\rho(\tau)) = \cL\big(\rho(\tau),\,m(\tau) + \nabla_\bG\rho(\tau)\big) - \cF(\rho(\tau))\qquad \text{on} \;\; [t, T).
\]
\end{remark}
\begin{lemma}\label{lem:cty_int} For every $\epsilon>0$, the value function $\cU$ is Lipschitz on $[0,T]\times \cP_\epsilon(\G)$ when we endow $\cP(\G)$ with the $\ell_2$-norm.
\end{lemma}
\begin{proof} Let $\epsilon>0$, $\mu_0, \mu_1 \in \cP_\epsilon(\G)$, $t_0, t_1 \in [0, T]$, and assume $t_0<t_1$. Let $(\rho, m)$ be the minimizer of $\cA_{t_0}^{T}(\bar \rho, \bar m) +\cU_{T}(\bar \rho({T}))$ over $\cC_{t_0}^{T}(\mu_0, \cdot)$. By Proposition \ref{prop:sufficient-min}, the range of $\rho$ is contained in $\cP_{c_\epsilon}(\G)$ for some $c_\epsilon>0$ depending only on $\epsilon$ and $T$, and $\|\phi\|_{\ell_\infty} \leq C/\epsilon$ for a constant $C$ independent of $\mu_0$ and $t_0$. By the first equation in \eqref{eq:mfg-withcH}, these uniform estimates imply $\|\dot \rho\|_{\ell_2} \leq M_\epsilon$ for some $M_\epsilon>0$ depending only on $\epsilon$.
By Corollary \ref{cor:continuous}, there exists $C_\epsilon>0$ such that $\cU(t, \cdot)$ is $C_\epsilon$--Lipschitz on $\cP_{c_\epsilon}(\G)$ for all $t \in [0, T]$. Increasing $C_\epsilon$ if necessary, Remark \ref{rem:diff_char_val} implies that $\cU(\cdot, \rho(\cdot))$ is $C_\epsilon$--Lipschitz on $[t_0, T]$. Thus
\begin{align*}
|\cU(t_0,\mu_0) - \cU(t_1,\mu_0)| &= |\cU(t_0,\rho(t_0)) - \cU(t_1,\rho(t_1)) + \cU(t_1,\rho(t_1)) - \cU(t_1,\rho(t_0))|\\
 &\leq  C_\epsilon(t_1-t_0)+C_\epsilon\|\rho(t_1)-\rho(t_0)\|_{\ell_2} \leq C_\epsilon(1+M_\epsilon)(t_1-t_0),
\end{align*}
and the result follows from the triangle inequality, noting that $|\cU(t_1,\mu_0)-\cU(t_1,\mu_1)|\le C_{\epsilon}\|\mu_0-\mu_1\|_{\ell_2}.$ 
\end{proof}

\begin{proposition}[Time differentiability of $\cU$] \label{prop:1differential} The time derivative of $\cU$ exists on $(0,T)\times \cP_0(\G)$ and the equation
\[
-\partial_t \cU + \cH\big(\cdot, -\nabla_{\cW}\cU\big)- \Delta_\text{ind}\cU + \cF=0
\]
holds pointwise on $(0,T)\times \cP_0(\G)$. Furthermore, $\partial_t \cU$, $\nabla_{\cW}\cU$, and $\Delta_\text{ind}\cU$ are bounded on $(0, T) \times \cP_\epsilon(\G)$ by a constant that depends only on $\epsilon$ and the data $\cU_T$, $\cL$.
\end{proposition}
\proof{} Fix $t_0 \in (0,T)$, $\epsilon>0$, and $\mu_0 \in \cP_\epsilon(\G)$. Given $h \in (0, T-t_0)$, we denote by $(\rho, m)$ the minimizer of $(\bar \rho, \bar m) \mapsto \cA_{t_0}^{T}(\bar \rho, \bar m) +\cU_{T}(\bar \rho({T}))$ over $\cC_{t_0}^{T}(\mu_0, \cdot)$. Note that $(\rho, m)|_{[t_0, t_0+h]}$ is the minimizer over the interval $[t_0, t_0+h]$. We have 
\[
\cU(t_0+h, \rho(t_0+h)) -\cU(t_0+h,\rho(t_0))=\int_{t_0}^{t_0+h} \Big(\delta_{\mu} \cU\big(t_0+h, \rho(\tau)\big), \dot \rho(\tau)\Big) d\tau:=U_1(h)+U_2(h)
\]
where, noting that by Corollary~\ref{cor:continuous} and Proposition~\ref{prop:sufficient-min}, $\delta_{\mu} \cU(t_0+h,\cdot)$ is Lipschitz and $\dot \rho$ is bounded,
\begin{align*}
U_1(h)&:= \int_{t_0}^{t_0+h} \Big(\delta_{\mu} \cU\big(t_0+h, \rho(\tau)\big)-\delta_{\mu}  \cU\big(t_0+h, \rho(t_0+h)\big), \dot \rho(\tau)\Big) d\tau = o(h),\\
\intertext{and, using the continuity equation and integration by parts,}
U_2(h)&:=\int_{t_0}^{t_0+h} \Big(\delta_{\mu}  \cU\big(t_0+h, \rho(t_0+h)\big), \dot \rho(\tau)\Big) d\tau = \Big(\nabla_\cW \cU\big(t_0+h, \rho(t_0+h)\big), \int_{t_0}^{t_0+h} m(\tau)d\tau\Big).
\end{align*}
By Proposition~\ref{thm:main2}, $\nabla_\cW \cU\big(\cdot, \rho(\cdot)\big)$ is continuous, and since $m$ is continuous,  
\begin{equation}\label{eq:aug29.2025.1}
\lim_{h \to 0^+}{U_2(h)\over h}=\big(\nabla_\cW \cU\big(t_0, \rho(t_0)\big), m(t_0)\big).
\end{equation}
We use Remark~\ref{rem:diff_char_val} to infer
\begin{align}
 \lim_{h \to 0^+}{\cU(t_0+h,\rho(t_0+h)) - \cU(t_0,\rho(t_0)) \over h}
= \cH\big(\rho(t_0), p(t_0)\big)- \big( m(t_0)+ \nabla_\bG \rho(t_0), p(t_0)\big) + \cF(\rho(t_0)).\label{eq:aug29.2025.2}
\end{align}
Using the decomposition
\begin{align*}
\cU(t_0+h,\mu_0)-\cU(t_0,\mu_0)  =& \cU(t_0+h,\mu_0) - \cU(t_0+h,\rho(t_0+h)) +  \cU(t_0+h,\rho(t_0+h))-\cU(t_0,\rho(t_0)) \\
=& -U_1(h)-U_2(h)+  \cU(t_0+h,\rho(t_0+h)) -\cU(t_0,\rho(t_0)),
\end{align*}
and, in light of \eqref{eq:aug29.2025.1} and \eqref{eq:aug29.2025.2}, we deduce
\begin{align}\label{eq:jan18.2026.3}
\lim_{h \to 0^+}{\cU(t_0+h,\mu_0)-\cU(t_0,\mu_0)\over h}=&\big(p(t_0), m(t_0)\big)+\cH\big(\rho(t_0), p(t_0)\big) - \big( m(t_0)+ \nabla_\bG \rho(t_0), p(t_0)\big) + \cF(\rho(t_0))\nonumber\\
= &\cH\big(\mu_0, p(t_0)\big)- \big( \nabla_\bG \mu_0, p(t_0)\big)+ \cF(\mu_0).
\end{align}
It remains to compute $\lim_{h \to 0^+}(\cU(t_0,\mu_0)-\cU(t_0-h,\mu_0))/h$. For $h \in (0, t_0)$, let $(\rho^h, m^h)$ denote the minimizer for $\cU(t_0-h, \mu_0)$ and set $p^h := D_m \cL(\rho^h, m^h + \nabla_\bG \rho^h)$. By Proposition~\ref{prop:sufficient-min}, $(\rho^h, m^h, p^h)$ and their derivatives are uniformly bounded and equicontinuous. Combined with the characterization of minimizers as the unique solution to \eqref{eq:mfg-withcH}, this implies $(\rho^h, m^h, p^h)|_{[t_0, T]} \to (\rho, m, p)$ in $C^1([t_0, T])$ as $h \to 0^+$; in particular, $p^h(t_0) \to p(t_0)$. The uniform bound on $\|\dot p^h\|_{\ell_\infty}$ hence implies $p^h(t_0-h) \to p(t_0)$.
Repeating the same decomposition as for the right derivative, we write
\[
\cU(t_0, \mu_0) - \cU(t_0-h, \mu_0) = -\bar U_1(h) -\bar U_2(h) + \cU(t_0, \rho^h(t_0)) - \cU(t_0-h, \rho^h(t_0-h)),
\]
where $\bar U_1(h) = o(h)$ by the same bounded-Hessian argument, and
\[
\bar U_2(h) = \Big(\nabla_\cW \cU\big(t_0, \rho^h(t_0)\big), \int_{t_0-h}^{t_0} m^h(\tau)\,d\tau\Big).
\]
By continuity, $\lim_{h \to 0^+} \bar U_2(h)/h = \big(-p(t_0), m(t_0)\big)$. Combined with the analogue of \eqref{eq:aug29.2025.2} for $(\rho^h, m^h, p^h)$ and the convergence $p^h(t_0-h) \to p(t_0)$, we obtain
\[
\lim_{h \to 0^+}{\cU(t_0 ,\mu_0)-\cU(t_0-h,\mu_0)\over h}  =  \cH\big(\mu_0, p(t_0)\big)- \big( \nabla_\bG \mu_0, p(t_0)\big)+ \cF(\mu_0).
\]
Together with \eqref{eq:jan18.2026.3}, this shows that $\cU(\cdot, \mu_0)$ is differentiable at $t_0$. Since $\rho(t_0)=\mu_0$ and, by Proposition~\ref{thm:main2}~(ii), $p(t_0)=-\nabla_{\cW}\cU(t_0, \mu_0)$, we conclude that
\begin{equation}\label{eq:aug30.2025.2}
\partial_t \cU(t_0, \mu_0)=  \cH\big(\mu_0, -\nabla_{\cW}\cU(t_0, \mu_0)\big)- \Delta_\text{ind}\cU(t_0,\mu_0) + \cF(\mu_0).
\end{equation} \endproof

\begin{proposition}\label{thm:main3} For any $\epsilon>0$, the functions $\partial_t \cU$ and $\delta_{\mu} \cU$ are Lipschitz on $(0, T) \times \cP_\epsilon(\G).$ In particular, $\cU$ is twice differentiable almost everywhere on $(0, T) \times \cP_0(\G).$ For each $\epsilon>0$, all partial second derivatives are bounded on $(0, T) \times \cP_\epsilon(\G)$ by a constant depending only on $\epsilon$.
\end{proposition}
\proof{} Combining Proposition \ref{thm:semi-concave} and Proposition \ref{prop:1differential}, the temporal and spatial derivatives of $\cU$ exist. Differentiating the identity in Proposition \ref{prop:1differential} directly shows that, pointwise and distributionally, 
\[
\nabla_{\cW} \partial_t \cU = \nabla_{\cW}\cH\big(\cdot, -\nabla_{\cW}\cU\big) - D_p \cH\big(\cdot, -\nabla_{\cW}\cU\big)\nabla^2_{\cW\cW}\cU + \nabla_{\cW} \cF +
\nabla_{\cW}\big(\nabla_\G \mu\big)\nabla_{\cW} \cU + \nabla^2_{\cW\cW} \cU \nabla_\G \mu.
\]
Hence if $\epsilon>0$, then $\|\nabla_{\cW} \partial_t \cU\|_{\ell_\infty}$ is bounded on $(0, T) \times \cP_\epsilon(\G)$ by a constant independent of $t$ that depends only on $\epsilon$. Furthermore, $\partial_t \nabla_{\cW} \cU=\nabla_{\cW} \partial_t \cU$ in the sense of distributions. Differentiating the identity in Proposition \ref{prop:1differential} once more, we conclude that, pointwise and distributionally, 
\[
\partial^2_{tt} \cU = -D_p \cH\big(\cdot, -\nabla_{\cW}\cU\big) \partial_t \nabla_{\cW}\cU - \Delta_\text{ind}\partial_t \cU.
\]
We conclude that $|\partial^2_{tt} \cU|$ is bounded on $(0, T) \times \cP_\epsilon(\G)$ by a constant depending only on $\epsilon$. \endproof

\begin{proposition}[Space--time $C^{2}$ regularity of $\cU$]\label{thm:C2-regularity}
The value function $\cU$ is $C^2$ on $[0,T]\times \cP_0(\G)$.
\end{proposition}
\proof{}
By the verification in Section~\ref{subsection:assumptions-g}, the hypotheses of Theorem~\ref{thm:main-mfg} hold, so $(t,\mu) \mapsto u(t,\mu):=\phi^{t,\mu}(t)$ is $C^1$ on $[0,T]\times\cP_0(\G)$. By Proposition~\ref{thm:main2}, $\nabla_\cW \cU(t,\mu)=\nabla_\G u(t,\mu)$. Since $\nabla_\cW \cU=\nabla_\G(\delta_\mu \cU)$ by definition, we have $\nabla_\G(\delta_\mu \cU - u)=0$. Connectedness of $(\bV,\G,\omega)$ implies that $\delta_\mu \cU - u$ is spatially constant, so $\delta_\mu \cU(t,\mu)=\Pi_{\bR^n_0} u(t,\mu)$. Since $u$ is $C^1$ and $\Pi_{\bR^n_0}$ is linear, $\delta_\mu \cU$ is $C^1$. By Proposition~\ref{prop:1differential},
\[
\partial_t \cU(t,\mu) = \cH(\mu,p(t,\mu))-(\nabla_\G \mu,p(t,\mu))+\cF(\mu),
\]
where $p(t,\mu):=-\nabla_\G u(t,\mu)$. Since $u$ is $C^1$, so is $p$. The functions $\cH$ and $\cF$ are smooth on the interior by assumptions \eqref{eq:dec20.2025.2} and \eqref{eq:assumpF}, and $(\mu,p)\mapsto(\nabla_\G\mu,p)$ is bilinear, so $\partial_t\cU$ is $C^1$.
We conclude that $\cU$ is $C^2$ on $[0,T]\times\cP_0(\G)$.
\endproof

\begin{proof}[Proof of Theorem~\ref{thm:main-hjb}]
Part~(i) combines Proposition~\ref{prop:action-minimizer}, which yields existence of a unique minimizer for $\cA_t^T(\cdot,\cdot)+\cU_T(\cdot(T))$ over $\cC_t^T(\mu,\cdot)$, with Proposition~\ref{prop:sufficient-min}, which characterizes that minimizer via $m=D_p\cH(\rho,-\nabla_\G\phi)-\nabla_\G\rho$, where $(\phi,\rho)$ is the unique classical solution of the forward--backward system~\eqref{eq:mfg-withcH}; the gradient identity $\nabla_\cW\cU(\cdot,\rho(\cdot))=\nabla_\G\phi$ is Proposition~\ref{thm:main2}(ii). For part~(ii), the $C^2$ regularity of $\cU$ on $[0,T]\times\cP_0(\G)$ is Proposition~\ref{thm:C2-regularity}, and the fact that $\cU$ satisfies the HJB equation is Proposition~\ref{prop:1differential}, combined with the $C^2$ regularity just established.
\end{proof}

%
%
%
%
\section{Master equation and Nash equilibria via Markov chains}\label{sec:NE-discreteMFG}
In this section, we assume that $(\cL, \cH, \cF, \cU_T)$ are as in Section \ref{sec:ClassicalHJE}. For $\mu \in \cP_0(\G)$ and $w \in \bS(n)$, we write
\[
\bar \cL(\mu, w):=\cL(\mu, m), \qquad m^{ij}:=\theta(\mu^i, \mu^j) w^{ij} \quad \forall (i,j) \in \bE.
\]
We keep $(H, g)$ as in \eqref{eq:define-all-functions}. We assume that $\cL(\cdot, m), \cH(\cdot, p)$ extend to smooth functions on $(0, 1)^n\times \bS(n):$
\[
\cL, \cH \in C^2\big((0,1)^n \times \bS(n)\big).
\]
Similarly, we assume that $\cF$ and $\cU_T$ have a smooth extension on $(0,1)^n.$ We set 
\[
L(\mu,  w):=D_\mu \bar \cL(\mu, w)-D_\mu \cF(\mu).
\]
%
%
%
%
\subsection{A variational formula for \texorpdfstring{$D_{\mu^i}\cH$}{D\_(mu i) H}}
\begin{definition}\label{defn:feb24.2026.1} We say that $(\bar \cL, \cH)$ satisfies the unique momentum  property if the following holds: whenever $\mu_1, \mu_2 \in \cP_0(\G)$, $\bar v, p_1, p_2\in \bS(n)$ and
$$
\bar \cL(\mu_i, \bar v)+\cH(\mu_i, p_i)=( \bar v, p_i)_{\mu_i}, \quad \forall i\in\{1,2\}
$$ 
then $ p_1=p_2.$
\end{definition}
The proof of the following useful lemma is left to the reader.
\begin{lemma}\label{lem:mar03.2026.2} Assume that $\bar v, p \in \bS(n)$ and $\mu \in (0,1)^n$ are such that $\bar \cL(\mu, \bar v)+\cH(\mu, p)=(\bar v, p)_\mu$. Then, 
\[
D_{v^{ij}} \bar \cL(\mu, \bar v)=\theta(\mu^i, \mu^j)p^{ij}, \quad D_{p^{ij}} \cH(\mu, p)=\theta(\mu^i, \mu^j)\bar v^{ij}
\]
and 
\[
D_{\mu^i}\bar \cL(\mu, \bar v)+ D_{\mu^i}\cH(\mu, p)=\sum_{j \in N(i)} \partial_1 \theta(\mu^i, \mu^j) \bar v^{ij} p^{ij}
\]
\end{lemma} 
For $a=(a^1, \cdots, a^{i-1}, a^{i+1}, \cdots, a^n) \in \bR^{n-1}$ we define $v[a] \in \bS(n)$ by setting  $v^{kl}[a]=0$ if $(k, l) \not \in \bE$ and for $(k, l) \in \bE$ we set
\[
v^{kl}[a]:=\left\{
     \begin{array}{lr}
      \hfill v^{kl}& \; \text{if} \quad i \not\in \{k, l\}\\
     \hfill v^{il}+ a^l &  \; \text{if} \hfill k =i\\
      \hfill  v^{ki}-a^k&  \; \text{if} \hfill l=i.
     \end{array}
   \right.
\]
Set 
\[
 \ell(a, v, i)\equiv  \ell(\mu, p, a, v, i):=  D_{\mu^i} \bar \cL(\mu, v[a]+\nabla_\G \log \mu)-\sum_{l \in N(i)}  \big(v[a]+\nabla_\G \log \mu\big)^{il} p^{il} \partial_1 \theta(\mu^i, \mu^l)
\]
and 
\[
K_i(\mu, p, v):=D_{\mu^i} \bar \cL(\mu, v+\nabla_\G \log \mu)-\sum_{l \in N(i)}  \big(v+\nabla_\G \log \mu\big)^{il} p^{il} \partial_1 \theta(\mu^i, \mu^l)
\]
so that 
\[
K_i(\mu, p, v[a])= \ell(\mu, p, a, v, i).
\]
Given $\mu \in \cP_0(\G),$  $v, p \in \bS(n)$ we often use the notation  
\[
\bar v:= v+\nabla_\G \log \mu
\] 
Observe that if $k<l$ and $(k, l)\in \bE$ then
\begin{equation}\label{eq:mar10.2026.1}
D_{v^{kl}}K_i(\mu, p, v)=
\left\{
     \begin{array}{lcr}
      \hfill D_{v^{il}} D_{\mu^i} \bar \cL(\mu, \bar v)-p^{il} \partial_1\theta(\mu^i, \mu^l) & \; \text{if} &k=i\\
     \hfill D_{v^{ki}} D_{\mu^i} \bar \cL(\mu, \bar v)-p^{ki} \partial_1\theta(\mu^k, \mu^i) & \; \text{if} &l=i\\
     \hfill  D_{v^{kl}} D_{\mu^i} \bar \cL(\mu, \bar v) &  \; \text{if}&k, l\not=i.
     \end{array}
   \right.
\end{equation}
\begin{lemma}\label{lem:feb24.2026.2new} 
Suppose that $(\bar \cL, \cH)$ satisfies the unique momentum  property, $i \in \bV$, $\mu \in \cP_0(\G),$ $D_{\mu^i} \bar \cL(\mu, \cdot)$  is convex and $v, p \in \bS(n)$ are such that $\bar \cL(\mu,  v+\nabla_\G \log \mu)+ \cH(\mu, p)=(v+\nabla_\G \log  \mu, p)_{\mu}.$ Then 
$$-\ell(0, v, i)=\max_a \{-\ell(a, v, i)\}= D_{\mu^i}\cH (\mu,  p).$$
If we further assume that  $D_{v^{kl}} D_{\mu^i} \bar \cL(\mu, \bar v)\equiv 0$ whenever $i\not \in \{k, l \}$ then 
\[
D_{\mu^i}\cH(\mu, p)= \sup_{w \in \bS(n)}  \sum_{l \in N(i)}  \big(w+\nabla_\G \log \mu\big)^{il} p^{il} \partial_1 \theta(\mu^i, \mu^l)- D_{\mu^i} 
\bar \cL(\mu, w+\nabla_\G \log \mu)
\]
and the maximum is attained when $w=v$. 
\end{lemma}
\proof{}  We set $\bar v:=v+\nabla_\G \log \mu$ and note that since $\ell(\cdot, v, i)$ is a convex function, its critical points are minimizers. To conclude the first part of the proof, it suffices to show that $D_{v^{ij}}K_i(\mu, p, v)=0$ for $j \in N(i).$ But if $j \in N(i)$,   then in light of \eqref{eq:mar10.2026.1}
\begin{equation}\label{eq:27Fev2026.1}
D_{v^{ij}}K_i(\mu, p, v)= D_{v^{ij}}D_{\mu^i} \bar \cL(\mu, \bar v)- p^{ij} \partial_1 \theta(\mu^i, \mu^j)=D_{\mu^i} D_{v^{ij}}\bar \cL(\mu, \bar v)-  p^{ij} \partial_1 \theta(\mu^i,  \mu^j).
\end{equation}
Let $p_t \in \bS(n)$ be such that 
$$
\bar \cL(\mu+te_i, \bar v)+\cH(\mu+te_i, p_t)=(p_t, \bar v)_{\mu+te_i}.
$$ 
Since $(\bar \cL, \cH)$ satisfies the unique momentum  property, $p_t=p$. We have that $(\mu+te_i, \bar v, p)$ is a critical point for $(\nu, \bar w, \pi) \to \bar \cL(\nu, \bar w)+\cH(\nu, p)-(\bar w, \pi)_{\nu}$, and so by Lemma \ref{lem:mar03.2026.2} 
\[
D_{v^{ij}} \bar \cL(\mu+te_i, \bar v)=\theta( \mu^i+t, \mu^j) p^{ij}.
\]
Thus, $D_{\mu^i}D_{v^{ij}} \bar \cL(\mu, \bar v)= \partial_1 \theta(\mu^i, \mu^j) p^{ij}$ and thus by \eqref{eq:27Fev2026.1}, $D_{v^{ij}}K_i(\mu, p, v)=0,$ which proves the first claim.

Now, using the fact that $K_i(\mu, p, \cdot)$ is a  convex function and since we know that  $D_{v^{ij}}K_i(\mu, p, v)=0$ for $j \in N(i)$, it remains to show that  $D_{v^{kl}}K_i(\mu, p, v)=0$ when $k\not =i$ and $l\not =i.$ But by assumption $D_{v^{kl}} D_{\mu^i} \bar \cL(\mu, \bar v)\equiv 0$, which together with \eqref{eq:mar10.2026.1} implies that $D_{v^{kl}}K_i(\mu, p, v)=0$. \endproof
\begin{example}Let $l_{ij}, h_{ij} \in C^2(\bR)$ be strictly convex functions such that $l_{ij}$ is the Legendre transform of $h_{ij}$. We assume that $h_{ij} \geq 0$ and there exist $C, C_1>0$ and $C_2 \in \bR$ such that $$C_1|s|^{p_0}+C_2\leq l_{ij}(s) \leq C(|s|^{p_0}+1).$$ We set 
\begin{equation}\label{eq:feb22.2026.5bisb}
\bar \cL(\mu, v)= {1\over 2} \sum_{(i, j) \in \bE} \theta(\mu^i, \mu^j)l_{ij}\big(v^{ij}\big), \qquad \cL(\mu, m)= {1\over 2}\sum_{(i, j) \in \bE}\theta(\mu^i, \mu^j)l_{ij}\bigg({m^{ij}\over \theta(\mu^i, \mu^j)}\bigg), 
\end{equation}
for $v, m \in \bS(n)$ and $\mu \in (0,1)^n$ and for $p \in \bS(n)$, we set 
\[
\cH(\mu, p):=\sup_{m \in \bS(n)} \big(m, p\big)-\cL(\mu, m) \equiv {1\over 2}\sum_{(i, j) \in \bE} \theta(\mu^i, \mu^j)h_{ij}\big(p^{ij}\big).
\]
Then $\cL$ is jointly convex, $D_{\mu^i} \bar \cL(\mu, \cdot)$ is convex and the pair $(\bar\cL, \cH)$ satisfies the unique momentum property. Furthermore, for any $i \in \bV,$ $D_{v^{kl}} D_{\mu^i} \bar \cL\equiv 0$ for all $(k, l)\in \bE$ such that $i\not \in \{k, l \}$.
\end{example}

%
\subsection{Markov chains and Nash equilibria}
In this section, we assume that $(\bar\cL, \cH)$ satisfies the unique momentum property and that $D_{\mu^i} \bar \cL(\mu, \cdot)$ is a convex function for all $i \in \{1, \cdots, n\}.$ When needed, we further assume that
\begin{equation}\label{eq:march12.2026.3}
D_{v^{kl}} D_{\mu^i} \bar \cL\equiv 0 \quad \forall (k, l)\in \bE, \; \forall i \in \{1, \cdots, n\}\; \text{such that}\;  i\not \in \{k, l \}.
\end{equation}
We call $Q \in C^1([0, T], \bR^{n \times n})$ a path of rate matrices if $Q^{ij}_t \geq 0$ for $i \not =j$ and $t \in [0, T]$ and 
\[
Q^{ii}_t=-\sum_{j \not =i}Q^{ij}_t \qquad (\forall i=1, \cdots, n).
\]
In this case, if $I_n \in \bR^{n \times n}$ is the identity matrix then for each $s \in [0,T]$, we denote by $\Psi(s, \cdot):[0,T] \to \bR^{n \times n}$ the unique solution to the system of ODEs
\begin{equation}\label{eq:feb10.2025.0sum}
\left\{
     \begin{array}{lr}
      \partial_t \Psi(s, t) = \Psi(s, t) Q(t)& \; \text{on} \; (0,T)\\
      \Psi(s, s) = I_n.&
     \end{array}
   \right.
\end{equation}
Given $f: \bV \to \bR$ and $0\leq s\leq t \le T$, we define $\hat Q_{s, t} f: \bV \to \bR$ by 
$$
\hat Q_{s, t} (f)= \Psi(s, t) f.
$$ 

We identify each $a \in \bR^n$ with the function on $\bV$ given by $a(e_i)=a^i$; equivalently, if $x \in \bV$ then $a(x)=(a, x)$.

The standard theory of Markov chains provides us with a probability space $(\Omega, \bP^\mu, \cF)$ and a filtration $\big(\cF_t\big)_{0 \leq t\leq T}$ such that $\cF_t$ is a sub-$\sigma$-algebra of $\cF$ and contains the $\bP^\mu$--null sets. Furthermore, there exists a Markov chain $(X_t)_{0 \leq t\leq T}$ which is $\big(\cF_t\big)_{0 \leq t\leq T}$--progressively measurable and such that the following proposition holds. 
\begin{proposition}\label{prop:feb05.2025.6sum} Set $\rho_t:=\mu \Psi(0, t)$. Then the  following hold.
\begin{enumerate}
\item[(i)] For every $\omega \in \Omega,$ $t \to X_t(\omega)$ is c\`adl\`ag (right continuous and has a left limit).
\item[(ii)] We have $\rho^i_t=\bP^\mu\big\{X_t=e_i \big\}$, $\dot \rho_t=\rho_t Q_t$ and $\rho_0=\mu$.
\item[(iii)] If $0\leq s<t$ then  $\bE\big[f(X_t)| \cF_s \big]=\bE\big[f(X_t)| X_s \big]=\big(\hat Q_{s,t} f\big)(X_s).$ 
\item[(iv)] $(M_t)_t:=(X_t-X_0 -\int_0^t X_\tau Q_\tau d\tau)_t$ is a martingale and for all $f \in C^1([0,T];\R^n)$,     
\begin{equation}\label{eq:Feb07.2026.9}
\big(f_t, X_t\big)=\big(f_0, X_0\big)+\int_0^t \big( \dot f_r, X_r\big)dr  + \int_0^t \big( f_r,X_r Q_{r} \big) dr +\int_0^t \big( f_r, dM_r\big)
\end{equation}
\end{enumerate}
\end{proposition}
\begin{definition}\label{defn:mar10.2026.5} Let $\big(\cF_t\big)_{0 \leq t\leq T}$ be a filtration on a probability space $(\Omega, \bP^\mu, \cF)$ such that each $\cF_t$ contains the $\bP^\mu$--null sets. If $(X_t)_{0 \leq t\leq T}$ is a Markov chain which is $\big(\cF_t\big)_{0 \leq t\leq T}$--progressively measurable and satisfies the properties in Proposition \ref{prop:feb05.2025.6sum}, we call $(X_t)_{0 \leq t\leq T}$ a Markov chain generated by the pair $(\mu, Q)$. 
\end{definition}
\begin{definition}\label{defn:mar10.2026.3} Let $\rho_* \in C^2([0, T]; \cP_0(\G)).$ We call $v \in C^1([0, T]; \bS(n))$ an admissible control for $\rho_*$ if for all $(i, j) \in \bE$ the functions
\begin{equation}\label{eq:mar10.2025.3}
Q^{ij}[v|\rho_*]:=\omega_{ij}-\sqrt{\omega_{ij}}\partial_1 \theta(\rho^i_*, \rho^j_*)\big(v + \nabla_\G \log \rho_*\big)^{ij}
\end{equation}
are non-negative on $[0, T].$ In this case, we set $Q^{ij}[v|\rho_*]=0$ if $(i, j) \not \in \bE$ and $i \not =j,$ and 
\begin{equation}\label{eq:mar10.2025.4}
Q^{ii}[v|\rho_*]:=-\sum_{j \in N(i)} Q^{ij}[v|\rho_*].
\end{equation}
\end{definition}
\begin{remark} Fix $\mu\in\cP_0(\G)$ and let $\rho_*\in C^2([0,T];\cP_0(\G))$ satisfy $\rho_{*,0}=\mu$. If $v\in C^1([0,T];\bS(n))$ is an admissible control for $\rho_*$, let $\Psi$ solve~\eqref{eq:feb10.2025.0sum} with $Q=Q[v|\rho_*]$ and set $\rho_t:=\mu\,\Psi(0,t)$. By Proposition~\ref{prop:feb05.2025.6sum}, $\dot \rho_t=\rho_t Q[v|\rho_*]_t$ and so, using the fact that $\partial_1 \theta(r, s)\equiv \partial_2 \theta(s, r)$ we conclude that
\[
\dot \rho^j= \sum_{i \in N(j)} \omega_{ij}(\rho^i-\rho^j)- \sum_{i \in N(j)} \sqrt{\omega_{ij}}\Big( \rho^i \partial_1 \theta(\rho^i_*, \rho^j_*)+\rho^j \partial_2 \theta(\rho^i_*, \rho^j_*)\Big)\big(v + \nabla_\G \log \rho_*\big)^{ij}.
\]
Thus, if $\rho=\rho_*$, since the identity $r \partial_1 \theta(r, s)+s \partial_2 \theta(r, s)\equiv \theta(r, s)$ holds then $\dot{\rho} + \nabla_\rho\cdot v = 0$. Indeed, 
\[
\dot \rho^j=\sum_{i \in N(j)} \omega_{ij}(\rho^i-\rho^j)-\sum_{i \in N(j)} \sqrt{\omega_{ij}} \theta(\rho^i, \rho^j)\big(v + \nabla_\G \log \rho\big)^{ij}
=-\sum_{i \in N(j)} \sqrt{\omega_{ij}} \theta(\rho^i, \rho^j) v^{ij}= - (\nabla_\rho\cdot v)^j.
\]
\end{remark}
\begin{remark}\label{rem:mar11.2026.1} Let $v$ be an admissible control for $\rho_*$ as in Definition~\ref{defn:mar10.2026.3}, and let $(X_t)_{0\leq t\leq T}$ be a Markov chain generated by $(\mu, Q[v|\rho_*])$. For any $f: [0,T] \to \bR^n$, if $i \in \{1, \cdots, n\}$ and $X_s=e_i$, then
\[
\Big(f_s, X_s Q[v|\rho_*]_s\Big)= \big(\Delta_{\G} f_s \big)^i+ \sum_{k\not=i}^n \big(\nabla_\G f_s \big)^{ik} \partial_1 \theta\Big((\rho_{*,s})^i, (\rho_{*,s})^k\Big)\big(v_s+\nabla_\G \log \rho_{*,s}\big)^{ik}.
\]
\end{remark}
For $\mu \in \cP_0(\G)$, we use the shorthand
\[
\rho_s^{0,\mu}:=\rho(s,0,\mu), \qquad p_s^{0,\mu}:=-\nabla_\G u\big(s,\rho_s^{0,\mu}\big).
\]
For each $s\in[0,T]$, let $\bar v_s^{0,\mu}\in \bS(n)$ satisfy
\[
\bar \cL\big(\rho_s^{0,\mu}, \bar v_s^{0,\mu}\big)+\cH\big(\rho_s^{0,\mu}, p_s^{0,\mu}\big)=\big(\bar v_s^{0,\mu}, p_s^{0,\mu}\big)_{\rho_s^{0,\mu}},
\]
and define
\[
v_s^{0,\mu}:=\bar v_s^{0,\mu}-\nabla_\G \log \rho_s^{0,\mu}.
\]

For $x \in \bV$, we write
\[
u\big(t,x,\rho^{0,\mu}\big):=u\big(t,\rho_t^{0,\mu}\big)(x), \qquad g(x,\rho):=g(\rho)(x), \qquad L(x,\rho,v):=L(\rho,v)(x).
\]
Let $X$ be a Markov chain as in Proposition \ref{prop:feb05.2025.6sum} and fix $t \in [0, T]$ and $i\in \{1, \cdots, n\}$. If $\sigma\{X_t=e_i\}$ denotes the $\sigma$--algebra generated by $\{X_t=e_i\} \subset \Omega$, then for $\cF$--measurable $Y: \Omega \to \bR$   we define the deterministic value
\[
\bE_{t, e_i}[Y]:=\bE\big[Y\big | \sigma\{X_t=e_i\}\big]\big|_{\{X_t=e_i\}}.
\]
\begin{definition}[Deviation yielded by admissible controls]\label{def:J} Fix $\mu\in \cP_0(\G)$, $t\in[0,T]$ and $i\in \{1, \cdots, n\}$. Given $v_* \in C^1([0, T]; \bS(n))$, let $\rho_* \in C^1([0, T]; \cP_0(\G))$ be the unique solution to $\dot \rho_*+\nabla_{\rho_*} \cdot v_*=0$.  If $v$ is an admissible control for $\rho_*$, define
\begin{equation}\label{eq:J_def}
J(t,i;v_*,v)=\mathbb E_{t,e_i}\left[g(X_T,\rho_{*,T})+\int_t^TL(X_s,\rho_{*,s},v_s +\nabla_\G \log \rho_{*\, s})\,ds\right],
\end{equation}
where $(X_t)_{0 \leq t\leq T}$ is a Markov chain generated by the pair $\big(\mu, Q[v|\rho_*]\big).$
\end{definition}
\begin{definition}\label{defn:mar10.2026.2} Let $\mu \in \cP_0(\G)$, $v_* \in C^1([0, T]; \bS(n))$ and suppose that the unique solution $\rho_* \in C^1([0, T]; \cP_0(\G))$ to $\dot \rho_*+\nabla_{\rho_*} \cdot v_*=0$ is such that $v_*$ is an admissible control for $\rho_*.$  We call $v_* \in C^1([0, T]; \bS(n))$ a Nash equilibrium for the system $(L, g, \mu)$ if for every $v \in C^1([0, T]; \bS(n))$ that is an admissible control for $\rho_*$ and every $i \in \{1, \cdots, n\}$ we have that
\[
J(0,i; v_*,v_*) \leq J(0,i;v_*,v).
\]
If instead we only have that 
\[
J(0,i;v_*,v_*) \leq J\big(0,i;v_*,v_*[a]\big),
\]
for all $a \in \bR^{n-1},$ we call $v_* \in C^1([0, T]; \bS(n))$ a restricted Nash equilibrium.
\end{definition}

We shall also use the pathwise cost
\[
\cJ(t;v_*,v):=g(X_T,\rho_{*,T})+\int_t^T L(X_s, \rho_{*, s} ,v_s+\nabla_\G \log \rho_{* \, s})\,ds,
\]
which is a random variable satisfying $J(t,i;v_*,v)=\bE_{t,e_i}[\cJ(t;v_*,v)]$.
\begin{lemma}\label{lem:mar12.2026} Assume $(\bar\cL, \cH)$ satisfies the unique momentum property and $D_{\mu^i} \bar \cL(\mu, \cdot)$ is a convex function for all $i \in \{1, \cdots, n\}.$ Let $\mu \in \cP_0(\G)$, and suppose that $v^{0, \mu}$ and $v \in C^1([0, T]; \bS(n))$ are admissible controls for $\rho^{0, \mu}.$ Let $(X_t)_{t}$ be a Markov chain generated by the pair $\big(\mu, Q[v|\rho^{0,\mu}]\big)$ and set $M_t:=X_t-X_0 -\int_0^t X_\tau Q[v|\rho^{0,\mu}]_\tau d\tau$. We have that
\begin{equation}\label{eq:mar12.2026.1b}
u\big(t, X_t,\rho^{0, \mu}\big)-\cJ(t;v^{0, \mu}, v)\leq -\int_t^T\Big(  u\big(s,\rho^{0, \mu}_s\big), dM_s\Big)
\end{equation}
if either $v=v^{0, \mu}[a]$ for some $a \in \bR^{n-1}$ or \eqref{eq:march12.2026.3} holds,
and equality holds in \eqref{eq:mar12.2026.1b} if  $v=v^{0, \mu}.$
\end{lemma}
\proof{} For $s \in [t, T]$, let $i(s)$ be such that $X_s=e_{i(s)}.$ We set
\[
{\rm Sum}:=
 \partial_s u\big(\cdot ,X,\rho^{0, \mu}\big) +\big(\delta_\rho u\big(\cdot,X,\rho^{0, \mu}\big),\dot\rho^{0, \mu}\big)+ \big(u\big(\cdot, \cdot,\rho^{0, \mu}\big)Q_\cdot^T, X \big) +L\big(X,\rho^{0, \mu}, v+\nabla_\G \log \rho^{0, \mu}\big).
\]
Since $u(T,\cdot)=g$, applying \eqref{eq:Feb07.2026.9} to $f_s:=u\big(s, \rho_s^{0, \mu}\big)$ we conclude that
\begin{equation}\label{eq:mar12.2026.1}
u\big(t, X_t,\rho^{0, \mu}\big)-\cJ(t; v^{0, \mu}, v)=- \int_t^T {\rm Sum}(s) ds - \int_t^T\big(  u(s,\rho^{0, \mu}_s), dM_s\big)
\end{equation}
We use Remark \ref{rem:mar11.2026.1} to infer 
\begin{align*}
 {\rm Sum}(s)=&L\big(X_s,\rho^{0, \mu}_s, v_s+\nabla_\G \log \rho^{0, \mu}_s\big)+ \partial_s u\Big(s,X_s,\rho^{0, \mu}_s\Big) +\Big(\nabla_\cW u(s,X_s,\rho^{0, \mu}_s), \bar v^{0, \mu}_s\Big)_{\rho^{0, \mu}_s}\\
&+\Delta_\text{ind} u(s,X_s,\rho^{0, \mu}_s)+ \Delta_{\G} u\big(s, X_s, \rho^{0, \mu}_s\big)\\
&+  \sum_{k\not={i(s)}}^n \Big(\nabla_\G u\big(s, \rho^{0, \mu}_s\big) \Big)^{{i(s)}k} \partial_1 \theta\Big((\rho_s^{0, \mu})^{i(s)}, (\rho_s^{0, \mu})^{k}\Big)\Big(v_s+\nabla_\G \log \rho^{0, \mu}_s\Big)^{{i(s)}k}
\end{align*}
Since $u$ satisfies the master equation  we infer  
\begin{align*}
 {\rm Sum}(s)=& L_{i(s)}\Big(\rho^{0, \mu}_s, v_s+\nabla_\G \log \rho^{0, \mu}_s\Big) +H_{i(s)} \Big(\rho^{0, \mu}_s, \nabla_\G u\big(s, \rho_s^{0, \mu}\big)\Big)\\
-&  \sum_{k\not={i(s)}}^n \Big(-\nabla_\G u\big(s, \rho^{0, \mu}_s\big) \Big)^{{i(s)}k} \partial_1 \theta\Big((\rho_s^{0, \mu})^{i(s)}, (\rho_s^{0, \mu})^{k}\Big)\Big(v_s+\nabla_\G \log \rho^{0, \mu}_s\Big)^{{i(s)}k}
\end{align*} 
This, together with \eqref{eq:mar12.2026.1} and Lemma \ref{lem:feb24.2026.2new}, proves the lemma. \endproof
\begin{theorem}[Nash equilibria]\label{thm:main-nash} Assume $(\bar\cL, \cH)$ satisfies the unique momentum property and $D_{\mu^i} \bar \cL(\mu, \cdot)$ is a convex function for all $i \in \{1, \cdots, n\}.$ Let $\mu \in \cP_0(\G)$. Assume moreover that $v^{0, \mu}$ is an admissible control for $\rho^{0, \mu}$. If \eqref{eq:march12.2026.3} holds, then $v^{0, \mu}$ is a Nash equilibrium for the system $(L, g, \mu).$ In the case \eqref{eq:march12.2026.3} fails, $v^{0, \mu}$ is merely a restricted Nash equilibrium for the system $(L, g, \mu).$
\end{theorem}
\begin{remark}\label{rmk:admissibility-torus}
The admissibility hypothesis in Theorem~\ref{thm:main-nash} is expected to hold on sufficiently fine nearest-neighbor discretizations of $\mathbb T^d$. With mesh size $h$ one has $\omega_{ij}=h^{-2}$ on each edge, so that
\[
Q^{ij}[v^{0,\mu}|\rho^{0,\mu}]
= h^{-2}-h^{-1}\partial_1 \theta\Big((\rho_s^{0,\mu})^i, (\rho_s^{0,\mu})^j\Big)(\bar v_s^{0,\mu})^{ij}.
\]
In the model family~\eqref{eq:feb22.2026.5bisb},
\[
(\bar v_s^{0,\mu})^{ij}=h_{ij}'\Big(-\nabla_\G u\big(s,\rho_s^{0,\mu}\big)^{ij}\Big),
\]
so an $n$-uniform bound on $\nabla_\G u$ yields an $n$-uniform bound on $\bar v^{0,\mu}$. Since the $h^{-2}$ term then dominates the drift correction, $Q^{ij}[v^{0,\mu}|\rho^{0,\mu}]\geq 0$ on every edge once $h$ is small enough, in agreement with the limiting MFG system whenever $\rho^{0,\mu}$ converges to a smooth positive solution. It remains an interesting open problem to establish uniform estimates for a rigorous proof of this limiting argument.
\end{remark}
\proof{} Let $v \in C^1([0, T]; \bS(n))$ be an admissible control for $\rho^{0, \mu}.$ Let $(X_t)_{t}$ be a Markov chain generated by the pair $\big(\mu, Q[v|\rho^{0,\mu}]\big)$ and set  $M_t:=X_t-X_0 -\int_0^t X_\tau Q[v|\rho^{0,\mu}]_\tau d\tau$. As in Definition \ref{defn:mar10.2026.5}, let $\big(\cF_t\big)_{0 \leq t\leq T}$ be a filtration on a probability space $(\Omega, \bP^\mu, \cF)$ such that each $\cF_t$ contains the $\bP^\mu$--null sets, and $(X_t)_{0 \leq t\leq T}$ is $\big(\cF_t\big)_{0 \leq t\leq T}$--progressively measurable. Since $(M_t)$ is a $(\cF_t)_t$--martingale and $\sigma\{X_t=e_i\} \subset \cF_t$, we have that 
\begin{equation}\label{eq:mar13.2026.4}
\bE\bigg[\int_t^T\Big(  u\big(s,\rho^{0, \mu}_s\big), dM_s\Big)\bigg| \sigma\{X_t=e_i\}\bigg]=\bE\Bigg[\bE\Big[\int_t^T\Big(  u\big(s,\rho^{0, \mu}_s\big), dM_s\Big)\Big| \cF_t\Big]\bigg| \sigma\{X_t=e_i\}\Bigg]=0.
\end{equation} 
By Proposition \ref{prop:feb05.2025.6sum}, $X_{0 \#}\bP^{\mu}=\mu$, so
\[
\bE_{0, e_i}\Big[u(0, X_0, \rho^{0, \mu})\Big]=u(0, e_i, \rho^{0, \mu}),
\]
which is independent of $v$. In light of Lemma \ref{lem:mar12.2026}, \eqref{eq:mar13.2026.4} with $t=0$ implies that if \eqref{eq:march12.2026.3} holds then
\[
J(0,i; v^{0, \mu}, v^{0, \mu}) = u(0, e_i,\rho^{0, \mu}) \leq J(0,i; v^{0, \mu}, v).
\]
Hence, $v^{0, \mu}$ is a Nash equilibrium for the system $(L, g, \mu).$ Similarly, in case \eqref{eq:march12.2026.3} fails, by Lemma \ref{lem:mar12.2026}, $v^{0, \mu}$ is a restricted Nash equilibrium for the system $(L, g, \mu).$
\endproof

\begin{ack*}
WG was supported by NSF grant DMS--2154578. SM and ZZ gratefully acknowledge the support of the UCLA Department of Mathematics through their Hedrick Math Fellowships. This project began while JW was a visiting faculty member at UCLA. WG thanks Y.~Gao, W.~Li, and J.-G.~Liu for fruitful discussions.
\end{ack*}

%
%
%
\bibliographystyle{amsplain}
\providecommand{\bysame}{\leavevmode\hbox to3em{\hrulefill}\thinspace}
\providecommand{\MR}{\relax\ifhmode\unskip\space\fi MR }
\providecommand{\MRhref}[2]{%
  \href{http://www.ams.org/mathscinet-getitem?mr=#1}{#2}
}
\providecommand{\href}[2]{#2}

\end{document}